\documentclass{article}
\usepackage[T1]{fontenc}
\usepackage[english]{babel}
\usepackage{calc}
\usepackage{mathrsfs}
\usepackage{amsmath}
\usepackage{amssymb}
\usepackage{amsthm}
\usepackage{array}
\usepackage[labelformat=simple]{subcaption}
\usepackage{float}
\usepackage{tabularx}
\usepackage{booktabs}

\DeclareMathOperator{\Endom}{End}

\DeclareMathOperator{\mo}{mod}

\DeclareMathOperator{\indes}{Ind}

\newcommand\Iso{\xrightarrow{\sim}}

\newcommand\Endo[1][U]{\Endom #1}

\newtheorem{theo}{Theorem}
\newtheorem{prop}[theo]{Proposition}
\newtheorem{coro}[theo]{Corollary}
\newtheorem{lema}[theo]{Lemma}
\theoremstyle{definition}
\newtheorem{defi}[theo]{Definition}
\theoremstyle{remark}
\newtheorem{rema}{Remark}
\newtheorem*{exam}{Example}
\usepackage{mathtools}
\usepackage{placeins}
\usepackage{xkeyval}
\usepackage{caption}

\DeclareCaptionLabelFormat{continued}{#1~#2 (Continued)}
\usepackage{array}
\usepackage{csquotes}
\usepackage{multicol}
\usepackage{enumitem}
\usepackage{tikz-cd}
\usepackage{biblatex}

\usetikzlibrary{
  calc,
  positioning, 
  matrix,
  arrows.meta,
  decorations.pathreplacing
}

\tikzset{
  copo/.style={draw,circle,inner sep=1.5pt},
  texto/.style={draw=none,inner sep=2pt},
  mimatriz/.style={
    anchor=base,
    align=center,
    text depth=0.8ex,
    text height=2.2ex,
    text width=2em,
  matrix of math nodes,
  nodes in empty cells
  },
  mimat/.style={
    matrix of math nodes,
    ampersand replacement=\&,
    text width=3.4em,
    text height=1.3em,
    text depth=0.6em,
    align=center
  },
  corto/.style={shorten >= 3pt,shorten <= 3pt}  
}

\setlist[enumerate,1]{label={\normalfont (\alph*)},ref={\normalfont (\alph*)}}

\newcommand\textdef[1]{\textit{#1}}

\newcommand\func[1]{\mathfrak{#1}}

\newcommand\obj{\mathcal{O}}

\newcommand\twomat[2]{\begin{bmatrix}#1\\#2\end{bmatrix}}

\DeclareMathOperator{\End}{End}
\DeclareMathOperator{\Hom}{Hom}
\DeclareMathOperator{\Dom}{Dom}
\DeclareMathOperator{\Ker}{Ker}
\DeclareMathOperator{\Ima}{Im}

\DeclareMathOperator{\Lrel}{lrel}
\DeclareMathOperator{\Rep}{rep}
\newcommand\lrel{\Lrel(k)}
\newcommand\lrelo{\Lrel_{1}(k)}
\newcommand\lrelt{\Lrel_{2}(k)}
\newcommand\lrelto{\Lrel^{2}_{1}(k)}
\newcommand\lreltt{\Lrel^{2}_{2}(k)}
\newcommand\repTk{\Rep(\carc{T},k)}

\newcommand\repFk{\Rep(\carc{F},k)}

\newcommand\repSk{\Rep(\mathcal{S},k)}

\newcommand\repDk{\Rep(\mathcal{D},k)}

\newcommand\repKk{\Rep(\mathcal{K},k)}

\newcommand\repCk{\Rep(\mathcal{C},k)}

\DeclareMathOperator{\Ind}{Ind}

\newcommand\IndF{\Ind\carc{F}}

\newcommand\IndS{\Ind\carc{S}}

\newcommand\carc[1]{\mathcal{#1}}

\newcommand\hormatrixl[7][]{%
\begin{tikzpicture}[#1]
\matrix[mimat,inner sep=0pt] (#7)
{
#2 \& #3 \\
};
\draw (#7-1-1.north west) -- (#7-1-1.south west);
\draw (#7-1-1.north east) -- (#7-1-1.south east);
\draw (#7-1-2.north east) -- (#7-1-2.south east);
\draw (#7-1-1.north west) -- (#7-1-2.north east);
\draw (#7-1-1.south west) -- (#7-1-2.south east);

\node[anchor=north,font=\footnotesize] at ([yshift=-0.5ex]#7.south) {$n\geq #4$};
\node[anchor=south,font=\footnotesize] at ([yshift=0.5ex]#7.north) {$d=#6$};
\end{tikzpicture}%
}

\newcommand\lrz{\scalebox{0.7}[0.8]{%
  $\begin{array}{c}0\,0\ldots 0\end{array}$}}
\newcommand\lroz{\scalebox{0.7}[0.8]{%
  $\begin{array}{c}1\,0\ldots 0\end{array}$}}

\newcommand\tttmatrixsimple[8][]{%
\begin{tikzpicture}[#1]
\matrix[mimat,inner sep=0pt,text width=3em] (mim)
{
#2 \& #3 \\
#4 \& #5 \\
};
\draw (mim-1-1.north west) -- (mim-2-1.south west);
\draw (mim-1-1.north east) -- (mim-2-1.south east);
\draw (mim-1-2.north east) -- (mim-2-2.south east);
\draw (mim-1-1.north west) -- (mim-1-2.north east);
\draw (mim-1-1.south west) -- (mim-1-2.south east);
\draw (mim-2-1.south west) -- (mim-2-2.south east);

\node[anchor=north,font=\footnotesize] 
  at ([yshift=-0.5ex]mim.south) 
  {$n\geq #6$};
\end{tikzpicture}%
}

\makeatletter
\define@cmdkey[PRE]{fam}{matname}{}
\define@cmdkey[PRE]{fam}{i}{}
\define@cmdkey[PRE]{fam}{ii}{}
\define@cmdkey[PRE]{fam}{iii}{}
\define@cmdkey[PRE]{fam}{iv}{}
\define@cmdkey[PRE]{fam}{v}{}
\define@cmdkey[PRE]{fam}{vi}{}
\define@cmdkey[PRE]{fam}{vii}{}
\define@cmdkey[PRE]{fam}{viii}{}
\define@cmdkey[PRE]{fam}{ix}{}
\define@cmdkey[PRE]{fam}{x}{}
\define@cmdkey[PRE]{fam}{xi}{}
\define@cmdkey[PRE]{fam}{xii}{}
\presetkeys[PRE]{fam}{
  matname=mimat,
  i=0,
  ii=0,
  iii=0,
  iv=0,
  v=0,
  vi=0,
  vii=0,
  viii=0,
  ix=0,
  x=0,
  xi=0,
  xii=0}{}
\newcommand\tetradmat[3]{
  \setkeys[PRE]{fam}{#1}
  \begin{tikzpicture}[
    ampersand replacement=\&,
    baseline=(current bounding box.center),
    #2
    ]
  \matrix[mimatriz,#3]
    (\cmdPRE@fam@matname)
    {
    \cmdPRE@fam@i 
      \& \cmdPRE@fam@iii
      \& \cmdPRE@fam@v 
      \& \cmdPRE@fam@vii \\
    \cmdPRE@fam@ii 
      \& \cmdPRE@fam@iv
      \& \cmdPRE@fam@vi 
      \& \cmdPRE@fam@viii \\
    };
  \draw 
    (\cmdPRE@fam@matname-1-1.north west) -- 
    (\cmdPRE@fam@matname-1-4.north east);  
  \draw 
    (\cmdPRE@fam@matname-1-1.south west) -- 
    (\cmdPRE@fam@matname-1-4.south east);  
  \draw 
    (\cmdPRE@fam@matname-2-1.south west) -- 
    (\cmdPRE@fam@matname-2-4.south east);  
  \foreach \Col in {1,2,3,4}
  {
  \draw 
    (\cmdPRE@fam@matname-1-\Col.north west) -- 
    (\cmdPRE@fam@matname-2-\Col.south west);
  }    
  \draw 
    (\cmdPRE@fam@matname-1-4.north east) -- 
    (\cmdPRE@fam@matname-2-4.south east);
  \end{tikzpicture}
}
\newcommand\tetradmatinfo[6]{
  \setkeys[PRE]{fam}{#1}
  \begin{tikzpicture}[
    ampersand replacement=\&,
    baseline=(current bounding box.center),
    #2
    ]
  \matrix[mimatriz,#3]
    (\cmdPRE@fam@matname)
    {
    \cmdPRE@fam@i 
      \& \cmdPRE@fam@iii
      \& \cmdPRE@fam@v 
      \& \cmdPRE@fam@vii \\
    \cmdPRE@fam@ii 
      \& \cmdPRE@fam@iv
      \& \cmdPRE@fam@vi 
      \& \cmdPRE@fam@viii \\
    };
  \draw 
    (\cmdPRE@fam@matname-1-1.north west) -- 
    (\cmdPRE@fam@matname-1-4.north east);  
  \draw 
    (\cmdPRE@fam@matname-1-1.south west) -- 
    (\cmdPRE@fam@matname-1-4.south east);  
  \draw 
    (\cmdPRE@fam@matname-2-1.south west) -- 
    (\cmdPRE@fam@matname-2-4.south east);  
  \foreach \Col in {1,2,3,4}
  {
  \draw 
    (\cmdPRE@fam@matname-1-\Col.north west) -- 
    (\cmdPRE@fam@matname-2-\Col.south west);
  }    
  \draw 
    (\cmdPRE@fam@matname-1-4.north east) -- 
    (\cmdPRE@fam@matname-2-4.south east);
  \node[anchor=north,font=\footnotesize] 
    at ([yshift=-0.5ex]\cmdPRE@fam@matname.south) 
    {$n\geq #4$};
  \end{tikzpicture}
}
\newcommand\tetradmatinfov[6]{
  \setkeys[PRE]{fam}{#1}
  \begin{tikzpicture}[
    ampersand replacement=\&,
    baseline=(current bounding box.center),
    #2
    ]
  \matrix[mimatriz,#3]
    (\cmdPRE@fam@matname)
    {
    \cmdPRE@fam@i 
      \& \cmdPRE@fam@iv
      \& \cmdPRE@fam@vii 
      \& \cmdPRE@fam@x \\
    \cmdPRE@fam@ii 
      \& \cmdPRE@fam@v
      \& \cmdPRE@fam@viii 
      \& \cmdPRE@fam@xi \\
    |[text height=1.4ex,text depth=0ex]|\cmdPRE@fam@iii 
      \& |[text height=1.4ex,text depth=0ex]|\cmdPRE@fam@vi
      \& |[text height=1.4ex,text depth=0ex]|\cmdPRE@fam@ix 
      \& |[text height=1.4ex,text depth=0ex]|\cmdPRE@fam@xii \\
    };
  \draw 
    (\cmdPRE@fam@matname-1-1.north west) -- 
    (\cmdPRE@fam@matname-1-4.north east);  
  \draw 
    (\cmdPRE@fam@matname-1-1.south west) -- 
    (\cmdPRE@fam@matname-1-4.south east);  
  \draw 
    (\cmdPRE@fam@matname-2-1.south west) -- 
    (\cmdPRE@fam@matname-2-4.south east);  
  \draw 
    (\cmdPRE@fam@matname-3-1.south west) -- 
    (\cmdPRE@fam@matname-3-4.south east);  
  \foreach \Col in {1,2,3,4}
  {
  \draw 
    (\cmdPRE@fam@matname-1-\Col.north west) -- 
    (\cmdPRE@fam@matname-3-\Col.south west);
  }    
  \draw 
    (\cmdPRE@fam@matname-1-4.north east) -- 
    (\cmdPRE@fam@matname-3-4.south east);
  \node[anchor=north,font=\footnotesize] 
    at ([yshift=-0.5ex]\cmdPRE@fam@matname.south) 
    {$n\geq #4$};
  \end{tikzpicture}
}

\newcommand\tetradmattworel[3]{
  \setkeys[PRE]{fam}{#1}
  \begin{tikzpicture}[
    ampersand replacement=\&,
    baseline=(current bounding box.center),
    #2
    ]
  \matrix[mimatriz,#3]
    (\cmdPRE@fam@matname)
    {
    \cmdPRE@fam@i 
      \& \cmdPRE@fam@iii
      \& \cmdPRE@fam@v 
      \& \cmdPRE@fam@vii \\
    \cmdPRE@fam@ii 
      \& \cmdPRE@fam@iv
      \& \cmdPRE@fam@vi 
      \& \cmdPRE@fam@viii \\
    };
  \draw 
    (\cmdPRE@fam@matname-1-1.north west) -- 
    (\cmdPRE@fam@matname-1-4.north east);  
  \draw 
    (\cmdPRE@fam@matname-1-1.south west) -- 
    (\cmdPRE@fam@matname-1-2.south east);  
  \draw 
    (\cmdPRE@fam@matname-2-1.south west) -- 
    (\cmdPRE@fam@matname-2-4.south east);  
  \foreach \Col in {1,2,3,4}
  {
  \draw 
    (\cmdPRE@fam@matname-1-\Col.north west) -- 
    (\cmdPRE@fam@matname-2-\Col.south west);
  }    
  \draw 
    (\cmdPRE@fam@matname-1-4.north east) -- 
    (\cmdPRE@fam@matname-2-4.south east);
  \end{tikzpicture}
}

\newcommand\tetradmatonerel[3]{
  \setkeys[PRE]{fam}{#1}
  \begin{tikzpicture}[
    ampersand replacement=\&,
    baseline=(current bounding box.center),
    #2
    ]
  \matrix[mimatriz,#3]
    (\cmdPRE@fam@matname)
    {
    \cmdPRE@fam@i 
      \& \cmdPRE@fam@iii
      \& \cmdPRE@fam@v 
      \& \cmdPRE@fam@vii \\
    \cmdPRE@fam@ii 
      \& \cmdPRE@fam@iv
      \& \cmdPRE@fam@vi 
      \& \cmdPRE@fam@viii \\
    };
  \draw 
    (\cmdPRE@fam@matname-1-1.north west) -- 
    (\cmdPRE@fam@matname-1-4.north east);  
  \draw 
    (\cmdPRE@fam@matname-1-1.south west) -- 
    (\cmdPRE@fam@matname-1-3.south east);  
  \draw 
    (\cmdPRE@fam@matname-2-1.south west) -- 
    (\cmdPRE@fam@matname-2-4.south east);  
  \foreach \Col in {1,2,3,4}
  {
  \draw 
    (\cmdPRE@fam@matname-1-\Col.north west) -- 
    (\cmdPRE@fam@matname-2-\Col.south west);
  }    
  \draw 
    (\cmdPRE@fam@matname-1-4.north east) -- 
    (\cmdPRE@fam@matname-2-4.south east);
  \end{tikzpicture}
}
\makeatother

\newcommand\tetradmatnoline[1]{%
  \begin{tikzpicture}[
    ampersand replacement=\&,
    baseline=(current bounding box.center)
    ]
  \matrix[mimatriz]
    (#1)
    {
      \& \& \& \\
      \& \& \& \\
    };
  \draw 
    (#1-1-1.north west) -- 
    (#1-1-4.north east);  
  \draw 
    (#1-2-1.south west) -- 
    (#1-2-4.south east);  
  \foreach \Col in {1,2,3,4}
  {
  \draw 
    (#1-1-\Col.north west) -- 
    (#1-2-\Col.south west);
  }    
  \draw 
    (#1-1-4.north east) -- 
    (#1-2-4.south east);
  \node[anchor=center] at (#1-1-1.south) {$A$};  
  \node[anchor=center] at (#1-1-2.south) {$B$};  
  \node[anchor=center] at (#1-1-3.south) {$C$};  
  \node[anchor=center] at (#1-1-4.south) {$D$};  
  \end{tikzpicture}%
}

\newcommand\tripmatrix[8][]{%
\begin{tikzpicture}[#1]
\matrix[mimat,inner sep=0pt] (#6)
{
#2 \& #3 \\
#4 \& \\
};
\draw (#6-1-1.north west) -- (#6-2-1.south west);
\draw (#6-1-1.north east) -- (#6-2-1.south east);
\draw (#6-1-2.north east) -- (#6-1-2.south east);
\draw (#6-1-1.north west) -- (#6-1-2.north east);
\draw (#6-1-1.south west) -- (#6-1-2.south east);
\draw (#6-2-1.south west) -- (#6-2-1.south east);

\node[anchor=north,font=\footnotesize] at ([yshift=-0.5ex]#6.south) {$n\geq #5$};
\end{tikzpicture}%
}

\newcommand\restr[2]{{% we make the whole thing an ordinary symbol
  \left.\kern-\nulldelimiterspace % automatically resize the bar with \right
  #1 % the function
  \vphantom{\big|} 
  \right|_{#2} % this is the delimiter
  }}

\newcommand\diamatrix[7][]{%
\begin{tikzpicture}[#1]
\matrix[mimat,nodes in empty cells,inner sep=0pt] (#7)
{
\& #2 \\
#3 \& \\
};
\draw (#7-1-1.south west) -- (#7-2-1.south west);
\draw (#7-1-1.north east) -- (#7-2-1.south east);
\draw (#7-1-2.north east) -- (#7-1-2.south east);

\draw (#7-1-2.north west) -- (#7-1-2.north east);
\draw (#7-1-1.south west) -- (#7-1-2.south east);
\draw (#7-2-1.south west) -- (#7-2-1.south east);

\node[anchor=north,font=\footnotesize] at ([yshift=-0.5ex]#7.south) {$n\geq #4$};
\end{tikzpicture}%
}

\newcommand\vermatrix[5][]{%
\begin{tikzpicture}[#1]
\matrix[mimat,inner sep=0pt] (#5)
{
#2 \\
#3 \\
};
\draw (#5-1-1.north west) -- (#5-2-1.south west);
\draw (#5-1-1.north east) -- (#5-2-1.south east);
\draw (#5-1-1.north west) -- (#5-1-1.north east);
\draw (#5-2-1.south west) -- (#5-2-1.south east);

\node[anchor=north,font=\footnotesize] at ([yshift=-0.5ex]#5.south) {$n\geq #4$};
\end{tikzpicture}%
}

\newcommand\tvermatrix[7][]{%
\begin{tikzpicture}[#1]
\matrix[mimat,inner sep=0pt] (#7)
{
#2 \& #3 \\
#4 \& #5 \\
};
\draw (#7-1-1.north west) -- (#7-2-1.south west);
\draw (#7-1-1.north east) -- (#7-2-1.south east);
\draw (#7-1-2.north east) -- (#7-2-2.south east);
\draw (#7-1-1.north west) -- (#7-1-2.north east);
\draw (#7-2-1.south west) -- (#7-2-2.south east);

\node[anchor=north,font=\footnotesize] at ([yshift=-0.5ex]#7.south) {$n\geq #6$};
\end{tikzpicture}%
}

\newcommand\Smatrix[5][]{%
\begin{tikzpicture}[#1]
\matrix[mimat,inner sep=0pt,text width=3em] (mim)
{
#2 \& #3 \\
#4 \& #5 \\
};
\draw (mim-1-1.north west) -- (mim-2-1.south west);
\draw (mim-1-1.north east) -- (mim-2-1.south east);
\draw (mim-1-2.north east) -- (mim-2-2.south east);
\draw (mim-1-1.north west) -- (mim-1-2.north east);
\draw (mim-1-1.south west) -- (mim-1-2.south east);
\draw (mim-2-1.south west) -- (mim-2-2.south east);
\end{tikzpicture}%
}

\newcommand\Dmatrix[4][]{%
\begin{tikzpicture}[#1]
\matrix[mimat,inner sep=0pt] (mim)
{
#2 \& #3 \\
#4 \&  \\
};
\draw (mim-1-1.north west) -- (mim-2-1.south west);
\draw (mim-1-1.north east) -- (mim-2-1.south east);
\draw (mim-1-2.north east) -- (mim-1-2.south east);
\draw (mim-1-1.north west) -- (mim-1-2.north east);
\draw (mim-1-1.south west) -- (mim-1-2.south east);
\draw (mim-2-1.south west) -- (mim-2-1.south east);
\end{tikzpicture}%
}

\newcommand\Kmatrix[3][]{%
\begin{tikzpicture}[#1]
\matrix[mimat,inner sep=0pt] (mim)
{
#2 \& #3 \\
};
\draw (mim-1-1.north west) -- (mim-1-1.south west);
\draw (mim-1-1.north east) -- (mim-1-1.south east);
\draw (mim-1-2.north east) -- (mim-1-2.south east);
\draw (mim-1-1.north west) -- (mim-1-2.north east);
\draw (mim-1-1.south west) -- (mim-1-2.south east);
\end{tikzpicture}%
}

\newcommand\CKmatrixi[3][]{%
\begin{tikzpicture}[#1]
\matrix[mimat,nodes in empty cells,text width=2.5em,inner sep=0pt] (mim)
{
\& #2 \\
#3 \& \\
};
\draw (mim-1-1.south west) -- (mim-2-1.south west);
\draw (mim-1-1.north east) -- (mim-2-1.south east);
\draw (mim-1-2.north east) -- (mim-1-2.south east);

\draw (mim-1-2.north west) -- (mim-1-2.north east);
\draw (mim-1-1.south west) -- (mim-1-2.south east);
\draw (mim-2-1.south west) -- (mim-2-1.south east);
\end{tikzpicture}%
}

\newcommand\knob[1]{$k$\nobreakdash-\hspace{0pt}#1}

\newcommand\inte[1][n]{\{1,\ldots,#1\}}

\newcommand\ile[1][n]{I^{\leftarrow}_{#1}}
\newcommand\iri[1][n]{I^{\rightarrow}_{#1}}
\newcommand\ido[1][n]{I^{\downarrow}_{#1}}
\newcommand\iup[1][n]{I^{\uparrow}_{#1}}

\title{Functorial embeddings associated with the Four Subspace Problem}
\author{Ivon Dorado\thanks{Departamento de Matemáticas, Universidad Nacional de Colombia, sede Bogotá} \and Gonzalo Medina\thanks{Departamento de Matemáticas, Universidad Nacional de Colombia, sede Manizales}}
\date{\today}

\begin{document}

\maketitle

\section*{Abstract}
We define a unified categorical framework for studying six subproblems arising from the classical Four Subspace Problem. For each subproblem, we construct a functor from its associated category to the category of representations of the quiver corresponding to the Four Subspace Problem. This approach gives a common structural setting for the six cases considered and allows a simultaneous and coherent analysis via functorial methods.

We prove that the six functors are additive and fully faithful, and we show that none of them is dense. As a consequence, each functor induces an equivalence between the corresponding source category and a well-identified full subcategory of the target category. These equivalences provide an effective mechanism for transferring classification results and structural properties, thereby clarifying the structural interrelations among the categories studied.\\[0.2cm]
\textbf{Keywords:} Four Subspace Problem; quiver representations; fully faithful functors; indecomposable objects.

\section*{Introduction}
The well-known Four Subspace Problem is closely connected to several classical classification problems in linear algebra and representation theory. These include problems involving quiver representations of types $\widetilde{A}_3$ and $\widetilde{A}_2$, the Kronecker pencil and its contragredient version, as well as classification problems for linear relations. Each of these has been studied extensively and occupies a central place in the theory of linear transformations and matrix problems. Despite their different formulations, they share a common structural nature and can be regarded as formalizations of a single underlying categorical setting. Understanding and exploiting this common structure is valuable, as it allows techniques and classification results obtained in one context to be transferred to others. 

In this paper we construct six additive, fully faithful functors from the categories associated with these subproblems into the category $\repFk$ of representations of the quiver $\carc{F}$ associated with the Four Subspace Problem. We show that their essential images are full subcategories of $\repFk$, and we provide intrinsic characterizations of these subcategories in terms of conditions on the associated linear maps. This yields a unified categorical framework in which all six subproblems can be studied simultaneously and, in particular, leads to a complete classification of the indecomposable objects in each subcategory.

The Four Subspace Problem consists of obtaining the classification, up to isomorphism, of four subspaces of a finite-dimensional vector space over a field. It was originally introduced and solved by~\textcite{GelPon1970}, and has since been addressed using a variety of methods (see, for example, \textcite{Naz1967,Naz1975,Bre1967,Bre1974,MedZav2004, AssSimSko2006-1,For2008}).

This problem is closely connected with several related classification problems. The case of a quiver of type $\widetilde{A}_{3}$ (the Euclidean, or affine, Dynkin diagram associated to $A_{3}$) deals with the classification of four linear operators defined between four finite-dimensional vector spaces and was first solved by Nazarova~\cite{Naz1967}. More recently, another solution was presented in~\cite{DorMed2025}. The case of type $\widetilde{A}_{2}$ similarly corresponds to the classification of three linear operators defined between three finite-dimensional vector spaces.

Another closely related problem is the Kronecker problem that involves the classification of pairs of linear transformations between two finite-dimensional vector spaces. A partial solution was given by Weierstrass in~\cite{Wei1868}, while a comprehensive treatment was  later provided by Kronecker in~\cite{Kro1990}. Over the last decades, various techniques have been used to propose different solutions to the original problem as well as some of its extensions (see, for example, \textcite{DorMed2023,Die46,Dev84,GabRoi1992,Ben95,Rin84,%
AusReiSma97,Zav2007,DmyFonRyb16}). A related variant is its contragredient version, where the domain and codomain of one of the two operators are interchanged. This was solved in~\cite{DobPon1965} by Dobrovol'skaya and Ponomarev.

In~\cite{Mac1961}, MacLane introduced linear relations under the denomination \textquote{additive relations}. In~\cite{Tow1971}, Towber studied the category formed by a single linear relation and obtained the classification of its indecomposable objects. In addition to the case of a single linear relation, that we generalize, we will also introduce the classification problem for pairs of linear relations from one vector space to another.  According to some literature (for example, \cite{GelPon1970}) such a classification was obtained by Ponomarev, although the original source is not readily available.

The paper is organized as follows. In Section~\ref{sec:FSPandsubproblems} we review the representation theory of the quiver associated with the Four Subspace Problem and describe the six subproblems under consideration. In Section~\ref{sec:functors} we construct the corresponding functors to the category of representations of $\carc{F}$, prove that they are additive and fully faithful, determine their images, and analyze the resulting categorical consequences. Finally, in Section~\ref{sec:matrices} we present explicit matrix descriptions and use them to obtain complete classifications of the indecomposable objects in each case.

\section{The Four Subspace Problem and its subproblems}
\label{sec:FSPandsubproblems}
Solving the Four Subspace Problem consists of obtaining the classification, up to isomorphism, of all the quadruples of subspaces of a finite-dimensional vector space over a field $k$.  This is equivalent to describing all indecomposable representations, up to isomorphism, of the partially ordered set $\carc{T}$, usually called \textdef{tetrad}, formed by four mutually incomparable points. 

The category $\repTk$ is a subcategory of the category $\repFk$ of representations of the quiver $\carc{F}$ which corresponds to a 
one-point extension of the tetrad $\carc{T}$ by adjoining a point $0$ and arrows from the original points to $0$.
\begin{figure}[H]
\centering
\begin{tikzpicture}[baseline=(current bounding box.center)]
\node[copo,label={below:$1$}] (1) {};
\node[copo,right=of 1,label={below:$2$}] (2) {};
\node[copo,right=of 2,label={below:$3$}] (3) {};
\node[copo,right=of 3,label={below:$4$}] (4) {};

\node[copo,right=4cm of 4,label={below:$0$}] (l0) {};

\node[copo,above left=of l0,label={above:$1$}] (l1) {}; 
\node[copo,below left=of l0,label={below:$2$}] (l2) {}; 
\node[copo,above right=of l0,label={above:$3$}] (l3) {}; 
\node[copo,below right=of l0,label={below:$4$}] (l4) {}; 
\node at ([yshift=1cm]{$(2)!0.5!(3)$ ) {$\carc{T}$}};
\node[yshift=2cm] at (l0) {$\carc{F}$};
\draw[->,corto] 
  (l1) -- 
  node[anchor=south,font=\footnotesize] {$\alpha$} 
  (l0);
\draw[->,corto] 
  (l2) -- 
  node[anchor=south,font=\footnotesize] {$\beta$} 
  (l0);
\draw[->,corto] 
  (l3) -- 
  node[anchor=south,font=\footnotesize] {$\gamma$} 
  (l0);
\draw[->,corto] 
  (l4) -- 
  node[anchor=south,font=\footnotesize] {$\delta$} 
  (l0);
\end{tikzpicture}
\caption{The tetrad $\carc{T}$ and its one-point extension $\carc{F}$.}
\end{figure}
The indecomposable representations of $\carc{F}$ consist of those arising from the tetrad together with the following representations, which correspond to the injective objects associated to the vertices $1$, $2$, $3$, and $4$:
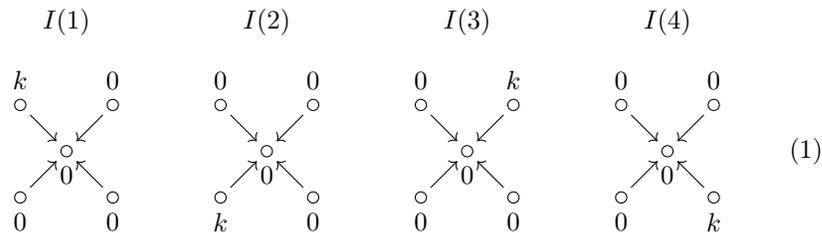
\begin{equation}
\label{equ:injective}
\begin{tikzpicture}[baseline={([yshift=-0.5cm]current bounding box.center)},node distance=0.5cm and 0.5cm]
% diagrama para I(1)
\node[copo,label={below:$0$}] (V10) {};
\node[copo,above left=of V10,label={above:$k$}] (V11) {}; 
\node[copo,below left=of V10,label={below:$0$}] (V12) {}; 
\node[copo,above right=of V10,label={above:$0$}] (V13) {}; 
\node[copo,below right=of V10,label={below:$0$}] (V14) {}; 
\draw[->,corto] 
  (V11) -- 
  (V10);
\draw[->,corto] 
  (V12) -- 
  (V10);
\draw[->,corto] 
  (V13) -- 
  (V10);
\draw[->,corto] 
  (V14) -- 
  (V10);
%
% diagrama para I(2)
\node[copo,right=2.5cm of V10,label={below:$0$}] (V20) {};
\node[copo,above left=of V20,label={above:$0$}] (V21) {}; 
\node[copo,below left=of V20,label={below:$k$}] (V22) {}; 
\node[copo,above right=of V20,label={above:$0$}] (V23) {}; 
\node[copo,below right=of V20,label={below:$0$}] (V24) {}; 
\draw[->,corto] 
  (V21) -- 
  (V20);
\draw[->,corto] 
  (V22) -- 
  (V20);
\draw[->,corto] 
  (V23) -- 
  (V20);
\draw[->,corto] 
  (V24) -- 
  (V20);
%
% diagrama para I(3)
\node[copo,right=2.5cm of V20,label={below:$0$}] (V30) {};
\node[copo,above left=of V30,label={above:$0$}] (V31) {}; 
\node[copo,below left=of V30,label={below:$0$}] (V32) {}; 
\node[copo,above right=of V30,label={above:$k$}] (V33) {}; 
\node[copo,below right=of V30,label={below:$0$}] (V34) {}; 
\draw[->,corto] 
  (V31) -- 
  (V30);
\draw[->,corto] 
  (V32) -- 
  (V30);
\draw[->,corto] 
  (V33) -- 
  (V30);
\draw[->,corto] 
  (V34) -- 
  (V30);
%
% diagrama para I(4)
\node[copo,right=2.5cm of V30,label={below:$0$}] (V40) {};
\node[copo,above left=of V40,label={above:$0$}] (V41) {}; 
\node[copo,below left=of V40,label={below:$0$}] (V42) {}; 
\node[copo,above right=of V40,label={above:$0$}] (V43) {}; 
\node[copo,below right=of V40,label={below:$k$}] (V44) {}; 
\draw[->,corto] 
  (V41) -- 
  (V40);
\draw[->,corto] 
  (V42) -- 
  (V40);
\draw[->,corto] 
  (V43) -- 
  (V40);
\draw[->,corto] 
  (V44) -- 
  (V40);
%
% names
\foreach \ver in {1,2,3,4}
  {
  \node[label={above:$I(\ver)$},yshift=1.3cm] at (V\ver 0) {};
  }
\end{tikzpicture}
\end{equation}
The indecomposable representations coming from the tetrad are displayed in Figure~\ref{fig:FSPsolution} in matrix form. These representations were obtained in Medina and Zavadskiy~\cite[Theorem 1]{MedZav2004}.
Together with the injective objects listed above, and using the equivalence between $\repTk$ and the full subcategory of socle-projective representations of $\repFk$ established by Bautista and Dorado in~\cite[Theorem 22]{BauDor2017}, we obtain the following classification result.
\begin{theo}
Given an arbitrary field $k$, all indecomposable representations of the quiver~$\carc{F}$ are determined, up to isomorphism, by the objects presented in Figure~\ref{fig:FSPsolution}, together with the indecomposable injective objects listed in~\eqref{equ:injective}.\hfill\qedsymbol
\end{theo}
Next we consider the representation type for the path algebra of the quiver $\carc{F}$ associated to the Four Subspace Problem. The representation type for each of the six subproblems  will be discussed in Section~\ref{sec:functors}.
\begin{theo}
\label{teo:finitegrowthF}
Let $A=k\carc{F}$ be the path algebra associated to the quiver $\carc{F}$. Then $A$ is tame, of finite growth.
\end{theo}
\begin{proof}
This is Exercise 4, Chapter XIX of \cite{SimSko2006-3}. For the path algebra $A=k\carc{F}$, let us consider the following $k[t]$-$A$-bimodule $N$:
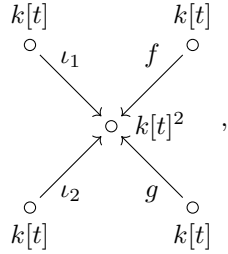
\[
\begin{tikzpicture}[
  node distance=2cm and 2cm,
  baseline=(current bounding box.center)
  ]
\node[copo,label={above:$k[t]$}] (1) {}; 
\node[copo,below=of 1,label={below:$k[t]$}] (2) {}; 
\node[copo,right=of 1,label={above:$k[t]$}] (3) {}; 
\node[copo,below=of 3,label={below:$k[t]$}] (4) {}; 
\coordinate (aux1) at ( $ (1)!0.5!(3) $ );
\coordinate (aux2) at ( $ (1)!0.5!(2) $ );
\node[copo,label={[label distance=3pt]right:$k[t]^{2}$}] at (aux1|-aux2) (0) {};
\draw[->,corto] 
  (1) -- 
  node[label={above:$\iota_{1}$}] {} 
  (0);
\draw[->,corto] 
  (2) -- 
  node[label={below:$\iota_{2}$}] {} 
  (0);
\draw[->,corto] 
  (3) -- 
  node[label={above:$f$}] {} 
  (0);
\draw[->,corto] 
  (4) -- 
  node[label={below:$g$}] {} 
  (0);
\end{tikzpicture},
\]
where $\iota_{1}$ and $\iota_{2}$ are the canonical inclusions, $f=\iota_{1}+\iota_{2}$ is the diagonal and $g$ is given by $g(q(t))=(tq(t),q(t))$, for all $q(t)\in k[t]$.

As left $k[t]$-module, $N$ is a finitely generated free module. Furthermore, if we now take an indecomposable $k[t]$-module $X$, then $X\simeq k[t]/(p^{s}(t))$, for an irreducible monic polynomial $p(t)\in k[t]$ with $r=s\cdot\deg p$. Moreover, the associated functor
\[
\widehat{N}=\text{---}\otimes_{k[t]} N_{A}
\colon
\indes k[t]\longrightarrow \mo A
\]
sends the module $X$ to the indecomposable $A$-module $\widehat{N}(X)$ of the form
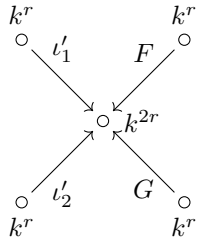
\[
\begin{tikzpicture}[
  node distance=2cm and 2cm,
  baseline=(current bounding box.center)
  ]
\node[copo,label={above:$k^{r}$}] (1) {}; 
\node[copo,below=of 1,label={below:$k^{r}$}] (2) {}; 
\node[copo,right=of 1,label={above:$k^{r}$}] (3) {}; 
\node[copo,below=of 3,label={below:$k^{r}$}] (4) {}; 
\coordinate (aux1) at ( $ (1)!0.5!(3) $ );
\coordinate (aux2) at ( $ (1)!0.5!(2) $ );
\node[copo,label={[label distance=2pt]right:$k^{2r}$}] at (aux1|-aux2) (0) {};
\draw[->,corto] 
  (1) -- 
  node[label={above:$\iota'_{1}$}] {} 
  (0);
\draw[->,corto] 
  (2) -- 
  node[label={below:$\iota'_{2}$}] {} 
  (0);
\draw[->,corto] 
  (3) -- 
  node[label={above:$F$}] {} 
  (0);
\draw[->,corto] 
  (4) -- 
  node[label={below:$G$}] {} 
  (0);
\end{tikzpicture}
\]
where $\iota'_{1}$ and $\iota'_{2}$ are the canonical inclusions, $F=\iota'_{1}+\iota'_{2}$ is the diagonal and $G$ is given by $G(v)=(T(v),v)$, for all $v\in k^{r}$, where $T\colon k^{r}\to k^{r}$ is the linear endomorphism whose matrix is the companion matrix of $p^{s}(t)$, for a monic and irreducible polynomial  $p(t)$.

On the other hand, as was proven in Medina and Zavadskij~\cite{MedZav2004} (see Figure~\ref{fig:FSPsolution}), $\indes_{d} A$ has infinitely many isomorphism classes of modules if and only if $d=2m$, for some integer $m\geq 1$ and $p(t)\notin\{t,t-1\}$ and so, for each integer $d\geq 1$, the family whose only element is the functor $\widehat{N}$ is an almost parametrizing family for $\indes_{d} A$. Hence, the algebra $A=k\carc{F}$ is tame, of finite growth ($\mu_{A}(d)\leq 1$, for all $d\geq 1$) and one-parametric.
\end{proof}

\subsection{\raggedright The quivers $\carc{S}$, $\carc{D}$, $\carc{K}$, and $\carc{C}$}
\label{sec:A3A2KC}
We will consider the four quivers $\carc{S}$, $\carc{D}$, $\carc{K}$, and $\carc{C}$ shown in Figure~\ref{fig:fourquivers}. Their representations correspond to the following configurations of vector spaces and linear operators.
\begin{figure}[H]
%A3
\begin{subcaptionblock}{0.5\textwidth}
\centering
\begin{tikzpicture}[node distance=1.3cm and 1.5cm,baseline=(current bounding box.center)]
\node[copo,label={left:$1$}] (S1) {};
\node[copo,below=of S1,label={left:$2$}] (S2) {};
\node[copo,right=of S1,label={right:$3$}] (S3) {};
\node[copo,below=of S3,label={right:$4$}] (S4) {};

\draw[->,corto] 
  (S3) -- 
    node[anchor=south] {$\alpha$} 
  (S1);
\draw[->,corto] 
  (S3) -- 
    node[near end,fill=white] {$\beta$} 
  (S2);
\draw[->,corto] 
  (S4) -- 
    node[near end,fill=white] {$\gamma$} 
  (S1);
\draw[->,corto] 
  (S4) -- 
    node[anchor=north] {$\delta$} 
  (S2);
\end{tikzpicture}
\caption{The quiver $\carc{S}$}
\label{fig:quiverS}
\end{subcaptionblock}%
%A2%
\begin{subcaptionblock}{0.5\textwidth}
\centering
% A2
\begin{tikzpicture}[node distance=1.3cm,baseline=(current bounding box.center)]
\node[copo,label={left:$1$}] (S1) {};
\node[copo,below=of S1,label={right:$2$}] (S2) {};
\node[copo,right=of S1,label={right:$3$}] (S3) {};

\draw[->,corto] 
  (S2) -- 
    node[anchor=east] {$\gamma$} 
  (S1);
\draw[->,corto] 
  (S3) -- 
    node[below right] {$\beta$} 
  (S2);
\draw[->,corto] 
  (S3) -- 
    node[above] {$\alpha$} 
  (S1);
\end{tikzpicture}
\caption{The quiver $\carc{D}$}
\label{fig:quiverD}
\end{subcaptionblock}\\[2ex]
% Kronecker
\begin{subcaptionblock}{0.5\textwidth}
\centering
\begin{tikzpicture}[node distance=1.5cm,baseline=(current bounding box.center)]
\node[copo,label={left:$1$}] (K1) {};
\node[copo,right=of K1,label={right:$2$}] (K2) {};

\draw[->,corto] 
  ( $ (K2) + (-3pt,3pt) $ ) -- 
    node[anchor=south] {$\alpha$} 
  ( $ (K1) + (3pt,3pt) $ );
\draw[->,corto] 
  ( $ (K2) + (-3pt,-3pt) $ ) -- 
    node[anchor=north] {$\beta$} 
  ( $ (K1) + (3pt,-3pt) $ );
\end{tikzpicture}
\caption{The quiver $\carc{K}$}
\label{fig:quiverK}
\end{subcaptionblock}%
% Contragredient
\begin{subcaptionblock}{0.5\textwidth}
\centering
\begin{tikzpicture}[node distance=1.5cm and 1.5cm,baseline=(current bounding box.center)]
\node[copo,label={left:$1$}] (C1) {};
\node[copo,right=of S1,label={right:$2$}] (C2) {};

\draw[->,corto] 
  ( $ (C2) + (-3pt,3pt) $ ) -- 
    node[anchor=south] {$\alpha$} 
  ( $ (C1) + (3pt,3pt) $ );
\draw[->,corto] 
  ( $ (C1) + (3pt,-3pt) $ ) -- 
    node[anchor=north] {$\beta$} 
  ( $ (C2) + (-3pt,-3pt) $ );
\end{tikzpicture}
\caption{The quiver $\carc{C}$}
\label{fig:quiverC}
\end{subcaptionblock}
\caption{Quivers associated to four of the subproblems considered.}
\label{fig:fourquivers}
\end{figure}
\subsubsection*{The quiver $\carc{S}$ of type $\widetilde{A}_{3}$}
A representation of is given by four vector spaces $V_{1}$, $V_{2}$, $V_{3}$, and $V_{4}$ together with linear maps
\[
f_{\alpha}\colon V_{3}\to V_{1},\qquad
f_{\beta}\colon V_{3}\to V_{2},\qquad
f_{\gamma}\colon V_{4}\to V_{1},\qquad
f_{\delta}\colon V_{4}\to V_{2}.
\]
This can be interpreted as the classification of pairs of linear maps from two source spaces into two target spaces. 

\subsubsection*{The quiver $\carc{D}$ of type $\widetilde{A}_{2}$}
A representation consists of three vector spaces $V_{1}$, $V_{2}$, $V_{3}$ and linear maps
\[
f_{\alpha}\colon V_{3}\to V_{1},\qquad
f_{\beta}\colon V_{3}\to V_{2},\qquad
f_{\gamma}\colon V_{2}\to V_{1}.
\]
forming a triangular configuration of operators. 
\subsubsection*{The Kronecker quiver $\carc{K}$}
A representation consists of two vector spaces $V_{1}$, $V_{2}$ and two linear operators:
\[
f_{\alpha}\colon V_{2}\to V_{1},\qquad
f_{\beta}\colon V_{2}\to V_{1}.
\]
The associated classification problem is the classical Kronecker problem: the classification of pairs of linear operators between two vector spaces up to simultaneous equivalence.
\subsubsection*{The contragredient Kronecker quiver $\carc{C}$}
A representation consists of two vector spaces $V_{1}$, $V_{2}$ and two linear maps:
\[
f_{\alpha}\colon V_{2}\to V_{1},\qquad
f_{\beta}\colon V_{1}\to V_{2}.
\]
The associated classification problem is often referred to as the classification of a pair of confronted operators.

The functorial relations that we will study in Section~\ref{sec:functors}, will allow us to view each of the category of representations of the quivers $\carc{S}$, $\carc{D}$, $\carc{K}$, and $\carc{C}$, together with the two involving linear relations presented below, embedded into the representation category of the quiver $\carc{F}$.

\subsection{Categories associated to linear relations}
Linear relations provide a natural generalization of linear operators. Unlike linear operators, they allow multivalued correspondences between vector spaces and relations whose domain is a proper subspace of the source space. They can be conveniently described as linear subspaces of a direct sum. In this section we introduce, based in the work of Towber~\cite{Tow1971}, the basic notions associated with linear relations and their representation categories. 
\begin{defi}
Given two vector spaces $V_{1}$ and $V_{2}$ over a field $k$, a \textdef{\knob{linear} relation from $V_{1}$ to $V_{2}$}, or simply a \textdef{linear relation from $V_{1}$ to $V_{2}$}, is a triple 
\[
\rho=(V_{1},V_{2},R), 
\]
where $R$ is any $k$\nobreakdash-subspace of $V_{1}\oplus V_{2}$. Equivalently, $R$ is a non-empty linear binary relation from $V_{1}$ to $V_{2}$ satisfying the following conditions: whenever $xRy$ and $x'Ry'$, then 
\[
(x+x')R(y+y') 
\]
and, for every $\alpha\in k$,
\[
(\alpha x)R(\alpha y).
\]

The \knob{linear} relation $(V_{1},V_{2},R)$ is called \textdef{finite-dimensional} if both $V_{1}$ and $V_{2}$ are finite-dimensional vector spaces over $k$. 

The \textdef{inverse relation} of a \knob{linear} relation and the \textdef{product} (or \textdef{composition}) of two \knob{linear} relations are defined in the usual way. It is worth noting that they are again linear relations. 
\begin{defi}
The \textdef{dual relation} of a \knob{linear} relation $\rho=(V_{1},V_{2},R)$ is the relation
\[
\rho^{\ast}=(V_{1}^{\ast},V_{2}^{\ast},R^{\ast}), 
\]
where $V_{1}^{\ast}=\Hom_{k}(V_{1},k)$ and $V_{2}^{\ast}=\Hom_{k}(V_{2},k)$ are the algebraic dual spaces of $V_{1}$ and $V_{2}$, respectively, and  $R^{\ast}\subseteq V_{1}^{\ast}\oplus V_{2}^{\ast}$ is defined by
\[
R^{\ast}=\{(f,g)\in V_{1}^{\ast}\oplus V_{2}^{\ast} \mid xRy \Rightarrow g(x)=f(y)\}.
\]
\end{defi}
\begin{defi}
Let $\rho=(V_{1},V_{2},R)$ and $\sigma=(W_{1},W_{2},S)$ be linear relations. A \textdef{morphism}
\[
f\colon\rho\to\sigma 
\]
is a pair $f=(f_{1},f_{2})$ of linear operators $f_{1}\colon V_{1}\to W_{1}$ and $f_{2}\colon V_{2}\to W_{2}$ 
such that 
\[
xRy
\quad\text{implies}\quad 
f_{1}(x)Sf_{2}(y). 
\]
The morphism $f=(f_{1},f_{2})\colon\rho\to\sigma$ is an \textdef{isomorphism} if both $f_{1}$ and $f_{2}$ are \knob{linear} isomorphisms and 
\[
xRy 
\quad\text{if and only if}\quad
f_{1}(x)Sf_{2}(y).
\]
In this case, we write $\rho\simeq\sigma$.

Given linear relations $\rho=(V_{1},V_{2},R)$ and $\sigma=(W_{1},W_{2},S)$, we define their \textdef{direct sum} by
\[
\rho\oplus\sigma=(V_{1}\oplus W_{1},V_{2}\oplus W_{2},R\oplus S).
\] 
A linear relation $\rho$ is \textdef{indecomposable} if $\rho\neq 0$ and $\rho\simeq\sigma\oplus\tau$ implies $\sigma=0$ or $\tau=0$. 
\end{defi}
\begin{defi}
Fix two vector spaces $V_{1}$ and $V_{2}$ over a field $k$. Taking \knob{linear} relations $\rho=(V_{1},V_{2},R)$ as objects and the morphisms defined above, we obtain the category \textdef{$\lrelt$ of \knob{linear} relations from $V_{1}$ to $V_{2}$}. 

The category $\lrelo$, or simply $\lrel$, of \textdef{\knob{linear} relations on $V_{1}$} is the full subcategory of $\lrelt$ whose objects are of the form $\rho=(V_{1},V_{1},R)$, which will be denoted simply by $\rho=(V_{1},R)$.

Next, we will consider the category $\lreltt$ of pairs of linear relations. In this category, objects are of the form 
\[
\rho=(V_{1},V_{2},R_{1},R_{2}),
\] 
where $V_{1}$ and $V_{2}$ are two finite-dimensional vector spaces over a field $k$, and $R_{1}$ and $R_{2}$ are linear relations from $V_{1}$ to $V_{2}$.   

A morphism $f\colon\rho\to\sigma$ between two objects $\rho=(V_{1},V_{2},R_{1},R_{2})$ and $\sigma=(W_{1},W_{2},S_{1},S_{2})$ is given by a pair $f=(f_{1},f_{2})$ of \knob{linear} operators $f_{1}\colon V_{1}\to W_{1}$ and $f_{2}\colon V_{2}\to W_{2}$ such that, for all $x,y\in V_{1}$, it holds 
\[
xR_{i}y 
\quad\text{implies}\quad 
f_{1}(x)S_{i}f_{2}(y),
\] 
for all $i\in\{1,2\}$. 
\end{defi}

\begin{exam}

\item Let $f\colon V_{1}\to V_{2}$ be a linear operator and let $\rho=(V_{1},V_{2},R)$, where $R$ is given by
\[
R=\{(x,f(x)) \mid x\in V_{1}\}.
\]
Then $\rho$ is a linear relation. A linear relation constructed in this way is called \textdef{associated with a linear operator}.

Not every \knob{linear} relation, however, is associated with a linear operator. For instance, let $V_{1}=k^{n+1}$, for some fixed integer $n\geq 0$. Define a relation $R$ on $V_{1}$ by
\[
R=
\bigl\lbrace 
  \bigl( 
  (0,\alpha_{1},\ldots,\alpha_{n}),(\alpha_{1},\ldots,\alpha_{n},0) 
  \bigr) 
  \,\bigm\lvert\, 
  \alpha_{i}\in k, %\\
  \text{ for all }i\in\inte
\bigr\rbrace.
\]
This, as the reader can easily confirm, defines a linear relation $\rho=(V_{1},R)$ on $V_{1}$ such that $\Dom R\neq V_{1}$ and, therefore, it is not associated to a linear operator.
\end{exam}
In~\cite{Tow1971}, Towber states that $\lrel$ is an additive category. We can go one step further and prove that it is, in fact, a Krull-Schmidt category. Before doing so, we state an auxiliary lemma.
\begin{lema}
\label{lem:idemsplit}
If $\rho=(V_{1},R)$ is an object of $\lrel$ and $e\in\Endo[\rho]$ is any idempotent, then there exist an object $\sigma=(W,S)$ and morphisms 
\[
\rho\xrightarrow{p}\sigma\xrightarrow{q}\rho
\] 
such that $qp=e$ and $pq=1_{\sigma}$.
\end{lema}
\begin{proof}
Let us consider any idempotent $e\in\Endo[\rho]$ and take $\sigma=(W,S)$, where $W=\Ima e$ and 
\[
S=\{ (e(x),e(y)) \mid (x,y)\in R \}.
\]
Furthermore, let us take $p\colon V_{1}\to W$ as the morphism induced by $e$ and $q\colon W\to V_{1}$ as the canonical inclusion from $W$ to $V_{1}$. The reader can verify that $S$ is a subspace of $W\oplus W$, and that $p$ and $q$ are morphisms in $\lrel$ from $\rho$ to $\sigma$ and from $\sigma$ to $\rho$, respectively. 

For every $x\in V_{1}$, we have $qp(x)=p(x)=e(x)$, so $qp=e$. For each $y\in W$, $pq(y)=e(y)=e(e(z))$, for some $z\in V_{1}$, thus $pq(y)=e^{2}(z)=e(z)=y$. Hence $pq=1_{\sigma}$. 
\end{proof}

\begin{theo}
\label{teo:lreliskrull}
For a given field $k$, $\lrel$ is a Krull-Schmidt category.
\end{theo}
\begin{proof}
Verification that $\lrel$ is an additive category is straightforward (the direct sum is the biproduct in this category).  

Now, let us establish that indecomposable linear relations have local endomorphism rings. Consider an indecomposable object $\rho=(V_{1},R)$. We are going to prove that the only idempotents in $\End\rho$ are $0$ and $1_{\rho}$. If we assume that there exists a non-trivial idempotent $e\in \Endo[\rho]$, taking $f=1_{\rho}-e$, we have  both $e$ and $f$ are idempotents different from $0$ and $1_{\rho}$. Thanks to Lemma~\ref{lem:idemsplit}, there exist non-null objects $\sigma_{1}=(W_{1},S_{1})$, $\sigma_{2}=(W_{2},S_{2})$, where $W_{1}$ and $W_{2}$ are different from $V_{1}$, and morphisms
\[
\rho\xrightarrow{p_{1}}\sigma_{1}\xrightarrow{q_{1}}\rho
\quad\text{and}\quad
\rho\xrightarrow{p_{2}}\sigma_{2}\xrightarrow{q_{2}}\rho
\]
such that
\[
q_{1}p_{1} = e,\qquad  
q_{2}p_{2} = f,\qquad
p_{1}q_{1} = 1_{\sigma_{1}},
\qquad\text{and}\qquad  
p_{2}q_{2} = 1_{\sigma_{2}}.
\]
Setting $\iota_{i}\colon\sigma_{i}\to \sigma_{1}\oplus \sigma_{2}$, the canonical inclusions, and $\pi_{i}\colon\sigma_{1}\oplus \sigma_{2}\to\sigma_{i}$, the canonical projections, for $i\in\{1,2\}$, if
\[
u=\iota_{1}p_{1} + \iota_{2}p_{2}
\colon \rho\to\sigma_{1}\oplus\sigma_{2}
\quad\text{and}\quad
v=q_{1}\pi_{1} + q_{2}\pi_{2}
\colon \sigma_{1}\oplus\sigma_{2}\to\rho,
\]
we obtain
\begin{gather*}
uv
=(\iota_{1}p_{1} + \iota_{2}p_{2})(q_{1}\pi_{1} + q_{2}\pi_{2})
=q_{1}p_{1} + q_{2}p_{2} = e+f =1_{\rho}
\\
\shortintertext{and}
vu
=(q_{1}\pi_{1} + q_{2}\pi_{2})(\iota_{1}p_{1} + \iota_{2}p_{2})
=\iota_{1}\pi_{1} + \iota_{2}\pi_{2} = 1_{\sigma_{1}\oplus\sigma_{2}},
\end{gather*}
which shows that $\rho\simeq\sigma_{1}\oplus\sigma_{2}$, so $\rho$ decomposes in a non-trivial manner. This contradiction shows that indeed, $0$ and $1_{\rho}$ are the only idempotents in $\Endo[\rho]$.  Since $\End\rho$ is a subalgebra of the finite-dimensional algebra $\End_{k}{V_{1}}$, this is equivalent to the fact that $\End\rho$ is a local ring (see~\textcite{Lam2013}, Corollary 19.19).
\end{proof}

\begin{theo}
For an arbitrary field $k$, the category $\lreltt$ is a Krull-Schmidt category.
\end{theo}
\begin{proof}
The proof follows the same structure as the one we presented for $\lrel$ in Theorem~\ref{teo:lreliskrull}, with the distinction that the endomorphism ring for a pair of linear relations from $V_{1}$ to $V_{2}$ is now a subset of $\End_{k}V_{1}\times \End_{k}V_{2}$.
\end{proof}

\begin{rema}
$\lrelt$ and $\lreltt$ are \emph{not} abelian categories: consider the morphism $1_{V_{1}\oplus V_{2}}=(1_{V_{1}},1_{V_{2}})$ from $(V_{1},V_{2},0)$ to $(V_{1},V_{2},V_{1}\oplus V_{2})$. The morphism $1_{V_{1}\oplus V_{2}}$ is clearly a monomorphism and an epimorphism, but it is not an isomorphism in $\lrelt$. Analogously, for $\lreltt$, consider the morphism $f=(1_{V_{1}},1_{V_{2}})$ from $(V_{1},V_{2},0,0)$ to $(V_{1},V_{2},V_{1}\oplus V_{2},V_{1}\oplus V_{2})$. $f$ is clearly both a monomorphism and an epimorphism, but it is not an isomorphism.
\end{rema}
The reader might wonder about the situation for the categories $\lrelt$, of one linear relation from one vector space to another, and $\lrelto$, of two linear relations on one vector space. It can be shown (see \textcite{Med2025}) that the first case corresponds to a problem of finite type, whereas the second one is of wild type.

\section{Six functors and their properties}
\label{sec:functors}

\subsection{Definition of the functors}
\label{sec:functorsdef}
We now define six covariant functors $\func{F}^{1}$, $\func{F}^{2}$, $\func{F}^{3}$, $\func{F}^{4}$, $\func{F}^{5}$, and $\func{F}^{6}$ from the categories associated with each of the six subproblems to the category $\repFk$, as illustrated below:
\[
\begin{tikzpicture}[baseline=(current bounding box.center)]
\node (0) {$\repFk$};
\node at (180:2.5cm) (1) {$\repCk$};
\node at (240:2.5cm) (2) {$\lrel$};
\node at (300:2.5cm) (3) {$\lreltt$};
\node at (0:2.5cm) (4) {$\repSk$};
\node at (60:2.5cm) (5) {$\repDk$};
\node at (120:2.5cm) (6) {$\repKk$};
\draw[->] 
  (1) --
    node[inner sep=1pt,label={above:$\func{F}^{4}$}] {} 
  (0);
\draw[->] 
  (2) --
    node[inner sep=1pt,label={[xshift=-7pt,yshift=-7pt]above:$\func{F}^{5}$}] {} 
  (0);
\draw[->] 
  (3) --
    node[inner sep=1pt,label={[xshift=10pt,yshift=-7pt]above:$\func{F}^{6}$}] {} 
  (0);
\draw[->] 
  (4) --
    node[inner sep=1pt,label={above:$\func{F}^{1}$}] {} 
  (0);
\draw[->] 
  (5) --
    node[inner sep=1pt,label={[xshift=-8pt,yshift=-7pt]above:$\func{F}^{2}$}] {} 
  (0);
\draw[->] 
  (6) --
    node[inner sep=1pt,label={[xshift=10pt,yshift=-7pt]above:$\func{F}^{3}$}] {} 
  (0);
\end{tikzpicture}
\]

\begin{enumerate}[label={\normalfont(\roman*)}]
%F1
\item The functor $\func{F}^{1}\colon\repSk\to\repFk$. Given a representation 
\[
V=(V_{1},V_{2},V_{3},V_{4},f_{\alpha},f_{\beta},f_{\gamma},f_{\delta})
\]
of $\carc{S}$ over $k$ we define $\func{F}^{1}(V)$ as follows. Let $V_{0}=V_{1}\oplus V_{2}$ and denote by $\iota_{1}\colon V_{1}\to V_{0}$ and $\iota_{2}\colon V_{2}\to V_{0}$ the canonical inclusions. We then set
\[
\func{F}^{1}(V)=
  (V_{0},V_{1},V_{2},V_{3},V_{4},
  \iota_{1},\iota_{2},\bar{f}_{\gamma},\bar{f}_{\delta}),
\]
where 
\[
\bar{f}_{\gamma}
=\iota_{1}f_{\alpha} + \iota_{2}f_{\beta}\quad\text{and}\quad
\bar{f}_{\delta}
=\iota_{1}f_{\gamma} + \iota_{2}f_{\delta}.
\]
For a morphism $l=(l_{1},l_{2},l_{3},l_{4})$ from $V$ to $W$ in $\repSk$, the functor assigns 
\[
\func{F}^{1}(l)=
(l_{1}\oplus l_{2},l_{1},l_{2},l_{3},l_{4})
\colon \func{F}^{1}(V)\to  \func{F}^{1}(W).
\] 

% F2
\item The functor $\func{F}^{2}\colon\repDk\to\repFk$. Let
\[
V=(V_{1},V_{2},V_{3},f_{\alpha},f_{\beta},f_{\gamma})
\]
be a representation of $\carc{D}$ over $k$. As before, write $V_{0}=V_{1}\oplus V_{2}$ and denote the canonical inclusions by $\iota_{1}\colon V_{1}\to V_{0}$ and $\iota_{2}\colon V_{2}\to V_{0}$. We define
\[
\func{F}^{2}(V)=
  (V_{0},V_{1},V_{2},V_{3},V_{2},
  \iota_{1},\iota_{2},\bar{f}_{\gamma},\bar{f}_{\delta}),
\]
where 
\[
\bar{f}_{\gamma}
=\iota_{1}f_{\alpha} + \iota_{2}f_{\beta} 
\quad\text{and}\quad
\bar{f}_{\delta}
=\iota_{1}f_{\gamma} + \iota_{2}.
\]
If $l=(l_{1},l_{2},l_{3})$ is a morphism from $V$ to $W$ in $\repDk$, we set 
\[
\func{F}^{2}(l)=
(l_{1}\oplus l_{2},l_{1},l_{2},l_{3},l_{2})
\colon \func{F}^{2}(V)\to  \func{F}^{2}(W).
\] 

% F3
\item The functor $\func{F}^{3}\colon\repKk\to\repFk$. Consider a representation 
\[
V=(V_{1},V_{2},f_{\alpha},f_{\beta})
\]
of $\carc{K}$ over $k$ we define $\func{F}^{3}(V)$ as follows. Setting again $V_{0}=V_{1}\oplus V_{2}$ with canonical inclusions $\iota_{1}\colon V_{1}\to V_{0}$ and $\iota_{2}\colon V_{2}\to V_{0}$, we take
\[
\func{F}^{3}(V)=
  (V_{0},V_{1},V_{2},V_{2},V_{2},
  \iota_{1},\iota_{2},\bar{f}_{\gamma},\bar{f}_{\delta}),
\]
where 
\[
\bar{f}_{\gamma}
=\iota_{1}f_{\alpha} + \iota_{2} \quad\text{and}\quad
\bar{f}_{\delta}
=\iota_{1}f_{\beta} + \iota_{2}.
\]
Given a morphism $l=(l_{1},l_{2})$ from $V$ to $W$ in $\repKk$, the induced morphism is
\[
\func{F}^{3}(l)=
(l_{1}\oplus l_{2},l_{1},l_{2},l_{2},l_{2})
\colon \func{F}^{3}(V)\to  \func{F}^{3}(W).
\] 

% F4
\item The functor $\func{F}^{4}\colon\repCk\to\repFk$. For a representation 
\[
V=(V_{1},V_{2},f_{\alpha},f_{\beta})
\]
of $\carc{C}$, let $V_{0}=V_{1}\oplus V_{2}$ and denote the canonical inclusions by $\iota_{1}$ and $\iota_{2}$. The representation $\func{F}^{4}(V)$ is defined by
\[
\func{F}^{4}(V)=
  (V_{0},V_{1},V_{2},V_{1},V_{2},
  \iota_{1},\iota_{2},\bar{f}_{\gamma},\bar{f}_{\delta}),
\]
with
\[
\bar{f}_{\gamma}
=\iota_{1} + \iota_{2}f_{\alpha} \quad\text{and}\quad
\bar{f}_{\delta}
=\iota_{1}f_{\beta} + \iota_{2}.
\]
For a morphism $l=(l_{1},l_{2})$ from $V$ to $W$ in $\repCk$, we set 
\[
\func{F}^{4}(l)=
(l_{1}\oplus l_{2},l_{1},l_{2},l_{1},l_{2})
\colon \func{F}^{4}(V)\to  \func{F}^{4}(W).
\] 

%lrel
\item The functor $\func{F}^{5}\colon\lrel\to\repFk$. Let $\rho=(V_{1},R)$ be a linear relation. Put $V_{0}=V_{1}\oplus V_{1}$ and let $\iota_{1}\colon V_{1}\to V_{0}$ and $\iota_{2}\colon V_{1}\to V_{0}$ be the canonical inclusions. We associate to $\rho$ the representation
\[
\func{F}^{5}(\rho)=
( V_{0},V_{1},V_{1},V_{1},R,
\iota_{1},\iota_{2},\iota_{1}+\iota_{2}, j),
\]
where $j\colon R\to V_{0}$ is the natural inclusion.

If $l\colon\rho\to\sigma$ is a morphism from $\rho=(V_{1},R)$ to $\sigma=(W_{1},S)$, then
\[
\func{F}^{5}(l)=
\bigl( 
l\oplus l,l,l,l,\restr{(l\oplus l)}{R}
\bigr)
\colon 
\func{F}^{5}(\rho)\to\func{F}^{5}(\sigma).
\]

%lreltt
\item The functor $\func{F}^{6}\colon\lreltt\to\repFk$. To each object $\rho=(V_{1},V_{2},R_{1},R_{2})$ of $\lreltt$, $\func{F}^{6}$ assigns the representation of $\carc{F}$ given by 
\[
\func{F}^{6}(\rho)=
(V_{0},V_{1},V_{2},R_{1},R_{2},
\iota_{1},\iota_{2},j_{1},j_{2}),
\]
where $V_{0}=V_{1}\oplus V_{2}$, $\iota_{1}\colon V{1}\to V_{0}$ and $\iota_{2}\colon V{2}\to V_{0}$ are the canonical inclusions, and $j_{1}\colon R_{1}\to V_{0}$, $j_{2}\colon R_{2}\to V_{0}$ are inclusions.

Given a morphism 
\[
l=(l_{1},l_{2})\colon \rho=(V_{1},V_{2},R_{1},R_{2})\to \sigma=(W_{1},W_{2},S_{1},S_{2})
\]
in $\lreltt$, the functor assigns the morphism 
\[
\func{F}^{6}(l)=
\bigl(
l_{1}\oplus l_{2},l_{1},l_{2},\restr{(l_{1}\oplus l_{2})}{R_{1}},\restr{(l_{1}\oplus l_{2})}{R_{2}}
\bigr),
\] 
between the corresponding objects in $\repFk$.
\end{enumerate}

\subsection{Basic properties of the functors}
\label{sec:sixfunctors}
In this section, we establish properties of the previously defined functors.
\begin{lema}
\label{lem:first2morphisms}
Let  $X_{1}$, $X_{2}$, $X'_{1}$, and $X'_{2}$, be finite-dimensional \knob{vector} spaces, with canonical inclusions $\iota_{i}\colon X_{i}\to X_{1}\oplus X_{2}$ and $\iota'_{i}\colon X'_{i}\to X'_{1}\oplus X'_{2}$, for $i\in\{1,2\}$. Suppose we are given linear maps $f_{0}$, $f_{1}$, and $f_{2}$ such that the following diagram commutes:
\[
\begin{tikzcd}[ampersand replacement=\&]
X_{1}\arrow[r,"\iota_{1}"]\arrow[d,swap,"f_{1}"] 
  \& X_{1}\oplus X_{2}\arrow[d,swap,"f_{0}"]
  \& X_{2}\arrow[l,swap,"\iota_{2}"]\arrow[d,swap,"f_{2}"]
\\   
X'_{1}\arrow[r,"\iota'_{1}"] 
  \& X'_{1}\oplus X'_{2}
  \& X'_{2}\arrow[l,swap,"\iota'_{2}"]
\end{tikzcd}
\]
Then $f_{0}=f_{1}\oplus f_{2}$.
\end{lema}
\begin{proof}
Take $x_{i}\in X_{i}$, for $i\in\{1,2\}$. Then
\begin{align*}
f_{0}(x_{1},x_{2}) 
&= f_{0}(\iota_{1}(x_{1}) + \iota_{2}(x_{2})) \\
&= f_{0}\iota_{1}(x_{1}) + f_{0}\iota_{2}(x_{2}) \\
&= \iota'_{1}f_{1}(x_{1}) + \iota'_{2}f_{2}(x_{2}) \\
&= (f_{1}(x_{1}) , f_{2}(x_{2})) \\
&= (f_{1}\oplus f_{2})(x_{1},x_{2}).
\end{align*}
This completes the proof.
\end{proof}

\begin{lema}
\label{lem:ffpreservesiso}
A fully faithful functor $\func{G}\colon\mathcal{E}\to\mathcal{F}$ preserves and reflects isomorphism.
\end{lema}
\begin{proof}
Let $f\colon V\to W$ be an isomorphism with inverse $g\colon W\to V$. Then, 
\begin{gather*}
\func{G}(f)\func{G}(g)
= \func{G}(fg)
= \func{G}(1_{W})
= 1_{\func{G}(W)}
\\\shortintertext{and}
\func{G}(g)\func{G}(f)
= \func{G}(gf)
= \func{G}(1_{V})
= 1_{\func{G}(V)},
\end{gather*}
so $\func{G}$ preserves isomorphisms. Let us now assume that $\func{G}V\simeq \func{G}W$. We then have mutually inverse isomorphisms $f\colon \func{G}V\to \func{G}W$ and $g\colon \func{G}W\to \func{G}V$. Since $\func{G}$ is full, there exist morphisms $f'\colon V\to W$ and $g'\colon W\to V$ such that $\func{G}(f')=f$ and $\func{G}(g')=g$. Furthermore, 
\[
\func{G}(1_{W})=
1_{\func{G}(W)}=
fg=
\func{G}(f')\func{G}(g')=
\func{G}(f'g')
\]
and, since $\func{G}$ is faithful, we get $1_{W}=f'g'$. In a similar way, one sees that $1_{V}=g'f'$, so $f'$ and $g'$ are mutually inverse morphisms and $V\simeq W$.
\end{proof}
In the following theorem, we will examine the main properties of the six functors we defined in the previous section.
\begin{theo}
\label{teo:propertiesoffunctors}
Functors $\func{F}^{1}$, $\func{F}^{2}$, $\func{F}^{3}$, $\func{F}^{4}$, $\func{F}^{5}$, and $\func{F}^{6}$ are fully faithful.
\end{theo}
\begin{proof}
We prove the claim for $\func{F}^{1}$. The proofs for the remaining cases are analogous. 

Let $V$ and $W$ be arbitrary objects  in $\repSk$. Consider the induced map
\[
\func{F}^{1}_{V,W}\colon
\Hom_{\repSk}(V,W) \to
\Hom_{\repFk}(\func{F}^{1}(V),\func{F}^{1}(W)).
\]
\textbf{Faithfulness.} Let $l=(l_{1},l_{2},l_{3},l_{4})$ and $l'=(l'_{1},l'_{2},l'_{3},l'_{4})$. If $\func{F}^{1}_{V,W}(l)=\func{F}^{1}_{V,W}(l')$, equality of components implies $l_{i}=l'_{i}$, for all $i\in\{1,2,3,4\}$. Hence $l=l'$ and therefore $\func{F}^{1}$ is faithful. 

\textbf{Fullness.} Let 
\[
h=(h_{0},h_{1},h_{2},h_{3},h_{4})\colon
\func{F}^{1}(V)\to \func{F}^{1}(W).
\]
be a morphism in $\repFk$. From the commutativity of the defining diagram we obtain
\[
h_{0}\iota_{1}=\iota'_{1}h_{1},\qquad
h_{0}\iota_{2}=\iota'_{2}h_{2},\qquad
h_{0}\bar{f}_{\gamma}=\bar{g}_{\gamma}h_{3},
\qquad\text{and}\qquad
h_{0}\bar{f}_{\delta}=\bar{g}_{\delta}h_{4}.
\]
From the first two equalities and Lemma~\ref{lem:first2morphisms}, we get $h_{0}= h_{1}\oplus h_{2}$. Substituing this into the third equality yields
\begin{align}
\begin{split}
\label{equ:surjectivea}
\iota'_{1}(h_{1}f_{\alpha}) + \iota'_{2}(h_{2}f_{\beta})
&= (\iota'_{1}h_{1}\pi_{1} + \iota'_{2}h_{2}\pi_{2})
(\iota_{1}f_{\alpha} + \iota_{2}f_{\beta}) \\
&= (h_{1}\oplus h_{2})(\iota_{1}f_{\alpha} + \iota_{2}f_{\beta}) \\
&= h_{0}\bar{f}_{\gamma} = \bar{g}_{\gamma}h_{3} \\
&= (\iota'_{1}g_{\alpha} + \iota'_{2}g_{\beta})h_{3} \\
&= \iota'_{1}(g_{\alpha}h_{3}) + \iota'_{2}(g_{\beta}h_{3}).
\end{split}
\end{align}
Analogously, from the fourth equality, we get
\begin{equation}
\label{equ:surjectiveb}
\iota'_{1}(h_{1}f_{\gamma}) + \iota'_{2}(h_{2}f_{\delta})
= \iota'_{1}(g_{\gamma}h_{4}) + \iota'_{2}(g_{\delta}h_{4}).
\end{equation}
From~\eqref{equ:surjectivea} and~\eqref{equ:surjectiveb}, we get  
\[
h_{1}f_{\alpha} = g_{\alpha}h_{3},\qquad
h_{2}f_{\beta} = g_{\beta}h_{3},\qquad
h_{1}f_{\gamma} = g_{\gamma}h_{4},\quad\text{and}\quad
h_{2}f_{\delta} = g_{\delta}h_{4}.
\]
Thus $h'=(h_{1},h_{2},h_{3},h_{4})$ defines a morphism $h'\colon V\to W$ in $\repSk$. Since clearly $\func{F}^{1}(h')=h$, we conclude that $\func{F}^{1}_{V,W}$ is surjective and therefore $\func{F}^{1}$ is full.
\end{proof}
\begin{rema}
In Theorem~\ref{teo:propertiesoffunctors} we proved that functors $\func{F}^{1}$, $\func{F}^{2}$, $\func{F}^{3}$, $\func{F}^{4}$, $\func{F}^{5}$, and $\func{F}^{6}$ are fully faithful. They are \emph{not}, however, dense. Let us assume, by contradiction, that $\func{F}^{i}$, for $i\in\{1,2,3,4,5,6\}$, is dense. With reference to the list in Figure~\ref{fig:FSPsolution}, if we consider an indecomposable representation $U$ of type $\mathrm{V}$, the density of $\func{F}^{i}$ would imply that $U\simeq \mathscr{F}^{i}(V)$, for some object $V$ in the domain category of $\func{F}^{i}$, where $V$ would necessarily be an indecomposable object. Therefore, $V$ would have a matrix presentation which should have one of the forms listed in Proposition~\ref{pro:bijections}. This is absurd, since $\mathscr{F}^{i}$ transforms all indecomposable objects to indecomposable representations of types which not include type $\mathrm{V}$, and different types are mutually non-isomorphic. 
\end{rema}
\begin{lema}
\label{lem:ffadditive}
Let $\func{G}\colon\mathcal{E}\to\mathcal{F}$ be a fully faithful functor between additive categories. Then $\func{G}$ is additive if and only if $\func{G}$ preserves zero objects.
\end{lema}
\begin{proof}
$\Rightarrow$) Since $\func{G}$ is additive, for any pair of objects $U$ and $V$ in $\mathcal{E}$, the induced map
\[
\func{G}_{U,V}\colon
\Hom_{\mathcal{E}}(U,V)
\to
\Hom_{\mathcal{F}}(\func{G}U,\func{G}V)
\]
is a group homomorphism and this implies that $\func{G}_{U,V}(0)=0$, for every zero morphism. Now, let $Z$ denote the zero object in $\mathcal{E}$. We have that $1_{Z}=0_{Z}$, where $0_{Z}$ is the zero endomorphism from $Z$ to $Z$, and from here we get $1_{\func{G}Z}=\func{G}1_{Z}=\func{G}0_{Z}=0_{\func{G}Z}$, so $\func{G}Z$ is the zero object in $\mathcal{F}$.

$\Leftarrow$)  Since $\func{G}$ preserves zero objects, it also preserves zero morphisms. Now, let us consider a pair of objects $U$ and $V$ from $\mathcal{E}$ and their biproduct $(U\oplus V,p_{1},p_{2},\iota_{1},\iota_{2})$ with projections $p_{1}$, $p_{2}$ and injections $\iota_{1}$, $\iota_{2}$. We then have
\begin{equation}
\begin{aligned}
\label{ecu:ffai}
p_{1}\iota_{1} &= 1_{U}, 
&
p_{2}\iota_{2} &= 1_{V}, 
&&&
p_{1}\iota_{2} &= 0, 
&
p_{2}\iota_{1} &= 0, 
&
&\text{and}\\
&&&& \makebox[0pt]{$\iota_{1}p_{1}  + \iota_{2}p_{2} = 1_{U\oplus V}$.}
\end{aligned}
\end{equation}
Since $\func{G}$ is a functor,  with the help of \eqref{ecu:ffai} and the fact that $\func{G}$ preserves zero morphisms we get
\begin{equation}
\begin{aligned}
\label{ecu:ffa1}
\func{G}(p_{1})\func{G}(\iota_{1}) 
&= 1_{\func{G}U},
&
\func{G}(p_{2})\func{G}(\iota_{2}) 
&= 1_{\func{G}V} \\
\func{G}(p_{1})\func{G}(\iota_{2}) 
&= 0,
& 
\func{G}(p_{2})\func{G}(\iota_{1}) 
&= 0.
\end{aligned}
\end{equation}
Now, consider the morphism
\[
\func{G}(\iota_{1})\func{G}(p_{1})  + \func{G}(\iota_{2})\func{G}(p_{2}),
\]
which belongs to $\Hom_{\mathcal{F}}(\func{G}(U\oplus V),\func{G}(U\oplus V))$. Since $\func{G}$ is full, there exists a morphism $h\in\Hom_{\mathcal{E}}(U\oplus V,U\oplus V)$ such that 
\begin{equation}
\label{ecu:ffa3}
\func{G}(h)=
\func{G}(\iota_{1})\func{G}(p_{1})  + \func{G}(\iota_{2})\func{G}(p_{2})=
\func{G}(\iota_{1}p_{1})  + \func{G}(\iota_{2}p_{2}).
\end{equation}
A direct calculation shows that
\begin{align*}
\func{G}(h\iota_{1}) &=
\func{G}(h)\func{G}(\iota_{1}) \\
&= \bigl( \func{G}(\iota_{1}p_{1})  + \func{G}(\iota_{2}p_{2})\bigr)\func{G}(\iota_{1}) \\\
&=\func{G}(\iota_{1}p_{1}\iota_{1})  + \func{G}(\iota_{2}p_{2}\iota_{1}) \\
&= \func{G}(\iota_{1}1_{U}) + \func{G}(0) \\
&= \func{G}(\iota_{1}).
\end{align*}
Analogously, $\func{G}(h\iota_{2}) = \func{G}(\iota_{2})$. Since $\func{G}$ is faithful, this implies $h\iota_{1}=\iota_{1}$ and $h\iota_{2}=\iota_{2}$. This means that both morphisms $1_{U\oplus V}$ and $h$ make the following coproduct diagram commute:
\[
\begin{tikzcd}[ampersand replacement=\&,row sep=large]
U\ar[dr,"\iota_{1}"]
  \ar[ddr,swap,"\iota_{1}",bend right=50] 
  \&\& 
  V\ar[dl,swap,"\iota_{2}"]
  \ar[ddl,"\iota_{2}",bend left=50]
\\ 
\& 
  U\oplus V
  \ar[d,shift left=3pt,"h"]
  \ar[d,shift right=3pt,swap,"1_{U\oplus V}"] 
  \& \\
\& U\oplus V \&
\end{tikzcd}
\]
and, by the universal property of the coproduct, we conclude $h=1_{U\oplus V}$. From this and \eqref{ecu:ffa3}, we obtain
\[
1_{\func{G}(U\oplus V)}=
\func{G}(1_{U\oplus V})=
\func{G}(h)=
\func{G}(\iota_{1})\func{G}(p_{1})  + \func{G}(\iota_{2})\func{G}(p_{2})
\]
which, with equalities~\eqref{ecu:ffa1}, show that 
\[
(\func{G}(U\oplus V),\func{G}(p_{1}),\func{G}(p_{2}),\func{G}(\iota_{1}),\func{G}(\iota_{2}))
\] 
is a biproduct for $\func{G}(U)$ and $\func{G}(V)$ (in particular, $\func{G}(U\oplus V)\simeq\func{G}(U)\oplus \func{G}(V)$). We have thus established that $\func{G}$ preserves biproducts and this implies that it is additive. See~\textcite[Theorem~3.11, Chapter 3]{Fre1964}.
\end{proof}
\begin{theo}
\label{teo:functorsareadditive}
Functors $\func{F}^{1}$, $\func{F}^{2}$, $\func{F}^{3}$, $\func{F}^{4}$, $\func{F}^{5}$, and $\func{F}^{6}$ are additive.
\end{theo}
\begin{proof}
By Theorem~\ref{teo:propertiesoffunctors},
 each functor is fully faithful. Moreover, from the definition it is clear that the zero object of the domain category is sent to the zero object of $\repFk$. Hence, by Lemma~\ref{lem:ffadditive}, each functor is additive.
\end{proof}
For a given representation 
\[
V=(V_{0},V_{1},V_{2},V_{3},V_{4},
  f_{\alpha},f_{\beta},f_{\gamma},f_{\delta})
\]
in $\repFk$, denote by $\eta_{V}$ the linear operator $\eta_{V}=f_{\alpha}\oplus f_{\beta}\colon V_{1}\oplus V_{2}\to V_{0}$.
\begin{lema}
\label{lem:etafgamma}
Let
\begin{align*}
V &= (V_{0},V_{1},V_{2},V_{3},V_{4},
  f_{\alpha},f_{\beta},f_{\gamma},f_{\delta})\qquad\text{and} \\
V' &= (V'_{1}\oplus V'_{2},V'_{1},V'_{2},V'_{3},V'_{4},
  \iota_{1},\iota_{2},f'_{\gamma},f'_{\delta})
\end{align*}
be isomorphic representations in $\repFk$, with $\iota_{i}\colon V'_{i}\to V'_{1}\oplus V'_{2}$ the canonical inclusions, for $i\in\{1,2\}$. The following holds:
\begin{enumerate}[label={\normalfont(\alph*)}]
\item\label{ite:eviso} $\eta_{V}$ is an isomorphism if and only if $\eta_{V'´}$ is.

\item\label{ite:fgamma} If $f'_{\gamma}=\iota_{1}A + \iota_{2}B$ 
for some linear maps $A$ and $B$, then
\[
f_{\gamma}
= f_{\alpha}(l_{1}^{-1}Al_{3}) + f_{\beta}(l_{2}^{-1}Bl_{3}).
\]
Analogously, if $f'_{\delta}=\iota_{1}C + \iota_{2}D$ 
for some linear maps $C$ and $D$, then
\[
f_{\delta}
= f_{\alpha}(l_{1}^{-1}Cl_{4}) + f_{\beta}(l_{2}^{-1}Dl_{4}).
\]
\end{enumerate}
\end{lema}
\begin{proof}
\ref{ite:eviso} The commutativity of the diagram for $l$ immediately implies that of the following diagram:
\[
\begin{tikzcd}[ampersand replacement=\&]
V_{1}\oplus V_{2}
  \arrow[r, "\eta_{V}"]\arrow[d,swap,"l_{1}\oplus l_{2}"]
\& 
V_{0}
  \arrow[d, "l_{0}"] 
\\
V'_{1}\oplus V'_{2}
    \arrow[r,swap,"\eta_{V'}"] 
\& 
V'_{1}\oplus V'_{2} \\
\end{tikzcd}
\] 
and from here, since the vertical maps are isomorphisms, $\eta_{V}$ is an isomorphism if and only if $\eta_{V'}$ is.

\ref{ite:fgamma} The commutativity of the diagram defining $l$ yields 
\[
l_{0}f_{\alpha}=\iota_{1}l_{1},
\qquad
l_{0}f_{\beta}=\iota_{2}l_{2},
\qquad\text{and}\qquad
l_{0}f_{\gamma}=f'_{\gamma}l_{3}.
\]
From these equalities we get
\[
f_{\gamma}
=l_{0}^{-1}(\iota_{1}A + \iota_{2}B)l_{3}
= l_{0}^{-1}\iota_{1}Al_{3} + l_{0}^{-1}\iota_{2}Bl_{3} \\
= f_{\alpha}(l_{1}^{-1}Al_{3}) 
  + f_{\beta}(l_{2}^{-1}Bl_{3}). 
\]
The proof  for $f_{\delta}$ is identical. 
\end{proof}
By Theorem~\ref{teo:propertiesoffunctors} and Theorem~\ref{teo:functorsareadditive}, each one of the functors $\func{F}^{1}$, $\func{F}^{2}$, $\func{F}^{3}$, $\func{F}^{4}$, $\func{F}^{5}$, and $\func{F}^{6}$ is additive and fully faithful, so each of the categories $\mathcal{A}_{1}=\repSk$, $\mathcal{A}_{2}=\repDk$, $\mathcal{A}_{3}=\repKk$, $\mathcal{A}_{4}=\repCk$, $\mathcal{A}_{5}=\lrel$, and $\mathcal{A}_{6}=\lreltt$ is equivalent to a full subcategory $\mathscr{C}_{1}$, $\mathscr{C}_{2}$, $\mathscr{C}_{3}$, $\mathscr{C}_{4}$, $\mathscr{C}_{5}$, and $\mathscr{C}_{6}$ of $\repFk$. 

In the following theorem we present an intrinsic characterization of the full subcategories $\mathscr{C}_{i}$, for all $i\in\{1,2,3,4,5,6\}$. Denote by $\obj_{i}$ the collection of objects from the category $\mathscr{C}_{i}$, for all $i\in\{1,2,3,4,5,6\}$.  
\begin{theo}
\label{teo:structural}
With the notation above, we have the following:
\begin{enumerate}[label={\normalfont(\alph*)}]
\item\label{ite:ob1} $V\in\obj_{1}$ if and only if $\eta_{V}$ is an isomorphism.

\item\label{ite:ob2} $V\in\obj_{2}$ if and only if $\eta_{V}$ is an isomorphism and there exist an isomorphism $\theta\colon V_{4}\Iso V_{2}$ and a linear map $\psi\colon V_{4}\to V_{1}$ such that $f_{\delta}=f_{\alpha}\psi + f_{\beta}\theta$.
\item\label{ite:ob3} $V\in\obj_{3}$ if and only if $\eta_{V}$ is an isomorphism and there exist isomorphisms $\zeta\colon V_{3}\Iso V_{2}$ and $\theta\colon V_{4}\Iso V_{2}$ and linear maps $\phi\colon V_{3}\to V_{1}$ and $\psi\colon V_{4}\to V_{1}$ such that 
\[
f_{\gamma}=f_{\alpha}\phi+ f_{\beta}\zeta
\qquad\text{and}\qquad 
f_{\delta}=f_{\alpha}\psi+ f_{\beta}\theta.
\]
\item\label{ite:ob4} $V\in\obj_{4}$ if and only if $\eta_{V}$ is an isomorphism and there exist isomorphisms $\epsilon\colon V_{3}\Iso V_{1}$ and $\theta\colon V_{4}\Iso V_{2}$ and linear maps $\chi\colon V_{3}\to V_{2}$ and $\psi\colon V_{4}\to V_{1}$ such that 
\[
f_{\gamma}=f_{\alpha}\epsilon+ f_{\beta}\chi
\qquad\text{and}\qquad 
f_{\delta}=f_{\alpha}\psi+ f_{\beta}\theta.
\]
\item\label{ite:ob5} $V\in\obj_{5}$ if and only if $\eta_{V}$ is an isomorphism and $f_{\delta}$ is an injective linear operator, and there exist isomorphisms $\epsilon\colon V_{3}\Iso V_{1}$ and $\zeta\colon V_{3}\Iso V_{2}$ such that $f_{\gamma}=f_{\alpha}\epsilon+ f_{\beta}\zeta$.
\item\label{ite:ob6} $V\in\obj_{6}$ if and only if $\eta_{V}$ is an isomorphism and and $f_{\delta}$ and $f_{\gamma}$ are injective linear maps.
\end{enumerate}
\end{theo}
\begin{proof}
We present the full proof in the cases~\ref{ite:ob3} and~~\ref{ite:ob5}. The remaining ones are analogous. We consider a representation
$V=(V_{0},V_{1},V_{2},V_{3},V_{4},
f_{\alpha},f_{\beta}f_{\gamma}f_{\delta})$ in $\repFk$.

\ref{ite:ob3} $\Rightarrow)$ Suppose $V\in\obj_{3}$.  Then $V\simeq \func{F}^{3}(W)$, for some $W=(W_{1},W_{2},g_{\alpha},g_{\beta})$
in $\repKk$. Hence there exists an isomorphism  $l=(l_{0},l_{1},l_{2},l_{3},l_{4})\colon V\to \func{F}^{3}(W)$, such that
\[ 
l_{0}f_{\alpha}=\iota_{1}l_{1},\quad
l_{0}f_{\beta}=\iota_{2}l_{2},\quad
l_{0}f_{\gamma}=(\iota_{1}g_{\alpha} + \iota_{2})l_{3},
\quad\text{and}\quad
l_{0}f_{\delta}=(\iota_{1}g_{\beta} + \iota_{2})l_{4}.
\]
Since $\eta_{\func{F}^{3}(W)}=\iota_{1}\oplus \iota_{2}$ is the identity on $W_{1}\oplus W_{2}$, Lemma~\ref{lem:etafgamma}\ref{ite:eviso} implies that $\eta_{V}$ is an isomorphism.

Applying Lemma~\ref{lem:etafgamma}\ref{ite:fgamma} with $A=g_{\alpha}$, $B=1_{V_{2}}$, $C=g_{\beta}$, and $D=1_{V_{2}}$, we obtain 
\[
f_{\gamma}
= f_{\alpha}(l_{1}^{-1}g_{\alpha}l_{3}) + f_{\beta}(l_{2}^{-1}l_{3})
\qquad\text{and}\qquad
f_{\delta}
= f_{\alpha}(l_{1}^{-1}g_{\beta}l_{4}) + f_{\beta}(l_{2}^{-1}l_{4}).
\]
Taking
\[
\phi =l_{1}^{-1}g_{\alpha}l_{3},\qquad
\psi =l_{1}^{-1}g_{\beta}l_{4},\qquad
\zeta =l_{2}^{-1}l_{3},
\qquad\text{and}\qquad
\theta =l_{2}^{-1}l_{4}
\]
yields $f_{\gamma}= f_{\alpha}\phi + f_{\beta}\zeta$ and $f_{\delta}= f_{\alpha}\psi + f_{\beta}\theta$. 

$\Leftarrow)$ Conversely, assume $\eta_{V}$ is an isomorphism and there exist isomorphisms $\zeta\colon V_{3}\Iso V_{2}$, $\theta\colon V_{4}\Iso V_{2}$ and linear maps $\phi\colon V_{3}\to V_{1}$ and $\psi\colon V_{4}\to V_{1}$ such that 
\[
f_{\gamma}=f_{\alpha}\phi+ f_{\beta}\zeta
\qquad\text{and}\qquad 
f_{\delta}=f_{\alpha}\psi+ f_{\beta}\theta.
\]
Let $W=(V_{1},V_{2},h_{\alpha},h_{\beta})$, with $h_{\alpha}=\phi\zeta^{-1}$ and $h_{\beta}=\psi\theta^{-1}$.
A direct verification shows that $V\simeq \func{F}^{3}(W)$. Hence, $V\in\obj_{3}$.

% prueba del caso (e)
\ref{ite:ob5} $\Rightarrow)$ Let $V\in\obj_{5}$. Then $V\simeq \func{F}^{5}(\rho)$, for some linear relation $\rho=(W,R)$ in $\lrel$. As above, an isomorphism $l$ satisfying the defining relations implies that $\eta_{V}$ is an isomorphism.

By Lemma~\ref{lem:etafgamma}\ref{ite:fgamma}, with $A=B=1_{W}$, we get 
\[
f_{\gamma}
= f_{\alpha}(l_{1}^{-1}l_{3}) + f_{\beta}(l_{2}^{-1}l_{3}).
\]
Setting $\zeta =l_{2}^{-1}l_{3}$ and $\theta =l_{2}^{-1}l_{4}$ gives $f_{\gamma}= f_{\alpha}\phi + f_{\beta}\zeta$. Moreover, from $f_{\delta}=l_{0}^{-1}jl_{4}$, we get that $f_{\delta}$ is injective. 

$\Leftarrow)$ Conversely, suppose $\eta_{V}$ is an isomorphism, $f_{\delta}$ is injective, and there exist isomorphisms $\epsilon\colon V_{3}\Iso V_{1}$ and $\zeta\colon V_{3}\Iso V_{2}$ such that $f_{\gamma}=f_{\alpha}\epsilon+ f_{\beta}\zeta$.

Define a linear relation $\rho=(V_{3},R)$, where $R$ is the image of the composition
\[
V_{4}\xrightarrow{f_{\delta}}
V_{0}\xrightarrow{\eta_{V}^{-1}}
V_{1}\oplus V_{2}\xlongrightarrow{\epsilon^{-1}\oplus \zeta^{-1}}
V_{3}\oplus V_{3}.
\]
Applying $\func{F}^{5}$ to $\rho$ and constructing the natural isomorphism $l$ yields $\func{F}^{5}(\rho)\simeq V$. Hence, $V\in\obj_{5}$.
\end{proof}
The structural characterizations given in Theorem~\ref{teo:structural} make it possible to compare the classes $\obj_{i}$ directly. As an immediate consequence, we obtain the following inclusions.
\begin{prop}
\label{pro:inclusions}
\[
\obj_{3},\obj_{4}\subseteq \obj_{2}\subseteq \obj_{1}
\qquad\text{and}\qquad
\obj_{5},\obj_{6}\subseteq\obj_{1}.
\]
\end{prop}
\begin{proof}
By Theorem~\ref{teo:structural}\ref{ite:ob2}--\ref{ite:ob6}, if $V\in\obj_{i}$, for $i\in\{2,3,4,5,6\}$, then $\eta_{V}$ is an isomorphism, hence by Theorem~\ref{teo:structural}\ref{ite:ob1} $V\in\obj_{1}$. 

Furthermore, by Theorem~\ref{teo:structural}\ref{ite:ob3}, if $V\in\obj_{3}$ then $\eta_{V}$ is an isomorphism and and there exist an isomorphism $\zeta\colon V_{3}\Iso V_{2}$ and a linear map $\phi\colon V_{3}\to V_{1}$ such that $f_{\gamma}=f_{\alpha}\phi+ f_{\beta}\zeta$. Hence, by Theorem~\ref{teo:structural}\ref{ite:ob2}, $V\in\obj_{2}$.

Analogously, by Theorem~\ref{teo:structural}\ref{ite:ob4}, if $V\in\obj_{4}$ then $\eta_{V}$ is an isomorphism and and there exist an isomorphism $\theta\colon V_{4}\Iso V_{2}$ and a linear map $\psi\colon V_{4}\to V_{1}$ such that $f_{\delta}=f_{\alpha}\psi+ f_{\beta}\theta$. Hence, by Theorem~\ref{teo:structural}\ref{ite:ob2}, $V\in\obj_{2}$.
\end{proof}
Proposition~\ref{pro:inclusions}
 exhibits a natural hierarchy among the six subcategories $\mathcal{C}_{i}$. In particular, $\mathcal{C}_{1}$ serves as a maximal ambient class containing all others, while $\mathcal{C}_{2}$ occupies an intermediate position capturing the common structure underlying $\mathcal{C}_{3}$ and $\mathcal{C}_{4}$. The inclusions $\obj_{3},\obj_{4}\subseteq \obj_{2}$ reflect the fact that the corresponding problems can be viewed as refinements of the structure encoded in $\obj_{2}$, whereas the subcategories $\mathcal{C}_{5}$ and $\mathcal{C}_{6}$ embed directly into $\mathcal{C}_{1}$, without factoring through $\mathcal{C}_{2}$.

Since the categories $\mathcal{C}_{i}$ are full subcategories of $\repFk$, each one of the inclusions $\obj_{i}\subseteq\obj_{j}$ in Proposition~\ref{pro:inclusions} induces a fully faithful functor $\mathcal{C}_{i}\to\mathcal{C}_{j}$. Transporting these along the equivalences $\mathcal{A}_{i}\simeq\mathcal{C}_{i}$ yields fully faithful functors as illustrated below:
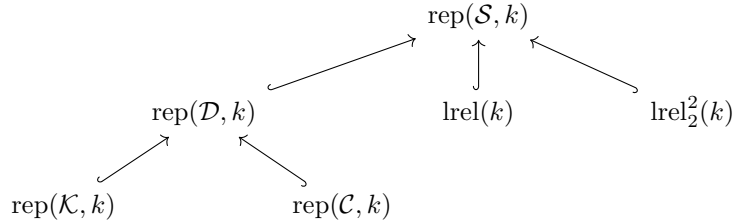
\[
\begin{tikzcd}[ampersand replacement=\&,column sep=0.5em]
\&\&\& \repSk \&\&\& 
\\
\& \repDk\arrow[urr,hook] 
  \&\& \lrel\arrow[u,hook] 
  \&\&\&\& \lreltt\arrow[ullll,hook] 
\\
\repKk\arrow[ur,hook] 
  \& \null \&  \repCk\arrow[ul,hook]
  \& \null \& \null \& \null
\end{tikzcd}
\]
The result in Theorem~\ref{teo:structural} also allows us to deduce additional categorical properties of the categories associated to the six subproblems under consideration, as illustrated in the following result and its corollary.
\begin{prop}
The categories $\mathcal{C}_{1}$, $\mathcal{C}_{2}$, $\mathcal{C}_{3}$, $\mathcal{C}_{4}$, $\mathcal{C}_{5}$, and $\mathcal{C}_{6}$ are closed under extensions. Consequently, each category inherits an exact structure from $\repFk$.
\end{prop}
\begin{proof}
For $\mathcal{C}_{1}$, $\mathcal{C}_{2}$, $\mathcal{C}_{3}$, and $\mathcal{C}_{4}$ this is immediate since they are abelian subcategories of $\repFk$, hence closed under extensions.

We shall provide the proof for $\mathcal{C}_{5}$. The argument for $\mathcal{C}_{6}$ is analogous. Let
\[
\begin{tikzcd}[
  ampersand replacement=\&,
  column sep=small
  ]
0\arrow[r] 
  \& U\arrow[r,"\iota"] 
  \& V\arrow[r,"\pi"] 
  \& W\arrow[r] 
  \& 0
\end{tikzcd}
\]
be a short exact sequence in $\repFk$, where $U$ and $W$ are objects in $\mathcal{C}_{5}$. To show that $V\in\obj_{5}$, we must verify that $V$ satisfies the characterizations given in Thorem~\ref{teo:structural}\ref{ite:ob5}. The argument is divided into three steps, using standard tools: the Five Lemma to show that $\eta_{V}$ is an isomorphism, a diagram chase for the injectivity of $f_{\delta,V}$, and a construction based on splittings to guarantee the existence of isomorphisms $\epsilon_{V}\colon V_{3}\Iso V_{1}$ and $\zeta_{V}\colon V_{3}\Iso V_{2}$ such that $f_{\gamma,V}=f_{\alpha}\epsilon_{V}+ f_{\beta,V}\zeta_{V}$.

Step 1: $\eta_{V}$ is an isomorphism. Since $U,V\in\obj_{5}$, the maps $\eta_{U}$ and $\eta_{W}$ are isomorphisms by hypothesis. We then have the following commutative diagram with exact rows:
\[
\begin{tikzcd}[
  ampersand replacement=\&,
  column sep=3em
  ]
0\arrow[r] 
  \& U_{1}\oplus U_{2}
    \arrow[r,"\iota_{1}\oplus\iota_{2}"]\arrow[d,swap,"\eta_{U}"] 
  \& V_{1}\oplus V_{2}
    \arrow[r,"\pi_{1}\oplus\pi_{2}"]\arrow[d,swap,"\eta_{V}"] 
  \& W_{1}\oplus W_{2}
    \arrow[r]\arrow[d,swap,"\eta_{W}"] 
  \& 0 \\
0\arrow[r] 
  \& U_{0}\arrow[r,"\iota_{0}"] 
  \& V_{0}\arrow[r,"\pi_{0}"] 
  \& W_{0}\arrow[r] 
  \& 0 \\
\end{tikzcd}
\]
Thus, $\eta_{V}$ is an isomorphism.

Step 2: $f_{\delta,V}$ is injective. Consider the following commutative diagram with exact rows:
\[
\begin{tikzcd}[
  ampersand replacement=\&,
  ]
0\arrow[r] 
  \& U_{4}\arrow[r,"\iota_{4}"]\arrow[d,swap,"f_{\delta,U}"] 
  \& V_{4}\arrow[r,"\pi_{4}"]\arrow[d,swap,"f_{\delta,V}"] 
  \& W_{4}\arrow[r]\arrow[d,swap,"f_{\delta,W}"] 
  \& 0 \\
0\arrow[r] 
  \& U_{0}\arrow[r,"\iota_{0}"] 
  \& V_{0}\arrow[r,"\pi_{0}"] 
  \& W_{0}\arrow[r] 
  \& 0 \\
\end{tikzcd}
\]
By hypothesis, $f_{\delta,U}$ and $f_{\delta,W}$ are injective. 
Let $v\in V_{4}$ such that $f_{\delta,V}(v)=0$. By the commutativity of the diagram, we have
\[
0=\pi_{0}f_{\delta,V}(v)=f_{\delta,W}\pi_{4}(v).
\]
The injectivity of $f_{\delta,W}$ implies $\pi_{4}(v)=0$, which means $v\in\Ker\pi_{4}=\Ima\iota_{4}$. Thus $v=\iota_{4}(u)$, for some $u\in U_{4}$. Utilizing the commutativity of the left square,
\[
\iota_{0}f_{\delta,U}(u)
= f_{\delta,V}\iota_{4}(u)
= f_{\delta,V}(v)
=0.
\]
Since $\iota_{0}$ and $f_{\delta,U}$ are both injective, it follows that $u=0$ and consequently, $v=\iota_{4}(u)=0$. Therefore, $f_{\delta,V}$ is injective.

Step 3: existence of the isomorphisms $\epsilon_{V}$ and $\zeta_{V}$. Since short exact sequences of quiver representations are pointwise exact, they split in the category of vector spaces. Thus, for each vertex $i\in\{0,1,2,3,4\}$, we can choose a section $s_{i}\colon W_{i}\to V_{i}$ and a retraction $r_{i}\colon V_{i}\to U_{i}$ such that $r_{i}\iota_{i}=1_{U_{i}}$, $\pi_{i}s_{i}=1_{W_{i}}$, and $\iota_{i}r_{i} + s_{i}\pi_{i}=1_{V_{i}}$.

Define the following maps:
\[
h_{\alpha}=r_{0}f_{\alpha,V}s_{1},\qquad
h_{\beta}=r_{0}f_{\beta,V}s_{2},\qquad
h_{\gamma}=r_{0}f_{\gamma,V}s_{3}.
\]
By hypothesis, there exist isomorphisms
\[
\epsilon_{U}\colon U_{3}\Iso U_{1},\qquad
\epsilon_{W}\colon W_{3}\Iso W_{1},\qquad
\zeta_{U}\colon U_{3}\Iso U_{2},\qquad
\zeta_{W}\colon W_{3}\Iso W_{2},
\]
such that $f_{\gamma,U}=f_{\alpha,U}\epsilon_{U} + f_{\beta,U}\zeta_{U}$ and $f_{\gamma,W}=f_{\alpha,W}\epsilon_{W} + f_{\beta,W}\zeta_{W}$.

Furthermore, the isomorphism $\eta_{U}\colon U_{1}\oplus U_{2}\to U_{0}$ induces an isomorphism between the corresponding $\Hom$-groups
\[
\eta_{U}^{\ast}\colon 
\Hom (W_{3},U_{1}\oplus U_{2})\Iso 
\Hom (W_{3},U_{0}).
\]
Therefore, we can uniquely choose a linear map $\sigma_{\epsilon}\oplus\sigma_{\zeta}$ from $W_{3}$ to $U_{1}\oplus U_{2}$ such that
\begin{equation}
\label{equ:theta}
\eta_{U}(\sigma_{\epsilon}\oplus\sigma_{\zeta})
=h_{\gamma}-h_{\alpha}\epsilon_{W}-h_{\beta}\zeta_{W}.
\end{equation}

We then define
\begin{align}
\begin{split}
\label{equ:epsilon}
\epsilon_{V} 
  &= \iota_{1}\epsilon_{U}r_{3}
  + s_{1}\epsilon_{W}\pi_{3}
  + \iota_{1}\sigma_{\epsilon}\pi_{3}, \\
\zeta_{V} 
  &= \iota_{2}\zeta_{U}r_{3}
  + s_{2}\zeta_{W}\pi_{3}
  + \iota_{2}\sigma_{\zeta}\pi_{3}.
\end{split}
\end{align}
A straightforward calculation shows that, with this definition, the following diagram commutes:
\[
\begin{tikzcd}[
  ampersand replacement=\&,
  ]
0\arrow[r] 
  \& U_{3}\arrow[r,"\iota_{3}"]\arrow[d,swap,"\epsilon_{U}"] 
  \& V_{3}\arrow[r,"\pi_{3}"]\arrow[d,swap,"\epsilon_{V}"] 
  \& W_{3}\arrow[r]\arrow[d,swap,"\epsilon_{W}"] 
  \& 0 \\
0\arrow[r] 
  \& U_{1}\arrow[r,"\iota_{1}"] 
  \& V_{1}\arrow[r,"\pi_{1}"] 
  \& W_{1}\arrow[r] 
  \& 0
\end{tikzcd}
\]
and, since by hypothesis $\epsilon_{U}$ and $\epsilon_{W}$ are isomorphisms, the Five Lemma implies that $\epsilon_{V}$ also is. Analogously, $\zeta_{V}$ is an isomorphism.

Furthermore,
\begin{align*}
f_{\alpha,V}
&= 1_{V_{0}}f_{\alpha,V}1_{V_{1}} \\
&= (\iota_{0}r_{0}+s_{0}\pi_{0})
  f_{\alpha,V}(\iota_{1}r_{1}+s_{1}\pi_{1}) \\
&= \iota_{0}r_{0}(f_{\alpha,V}\iota_{1})r_{1}
  + \iota_{0}(r_{0}f_{\alpha,V}s_{1})\pi_{1}
  + s_{0}\pi_{0}(f_{\alpha,V}\iota_{1})r_{1}
  + s_{0}(\pi_{0}f_{\alpha,V})s_{1}\pi_{1} \\
&= \iota_{0}r_{0}(\iota_{0}f_{\alpha,U})r_{1}
  + \iota_{0}h_{\alpha}\pi_{1}
  + s_{0}(f_{\alpha,W}\pi_{0})s_{1}\pi_{1} \\
&= \iota_{0}f_{\alpha,U}r_{1}
  + \iota_{0}h_{\alpha}\pi_{1}
  + s_{0}f_{\alpha,W}\pi_{1}.
\end{align*}
And, analogously,
\begin{align*}
f_{\beta,V}
&= \iota_{0}f_{\beta,U}r_{2}
  + \iota_{0}h_{\beta}\pi_{2}
  + s_{0}f_{\beta,W}\pi_{2},\\
f_{\gamma,V}
&= \iota_{0}f_{\gamma,U}r_{3}
  + \iota_{0}h_{\gamma}\pi_{3}
  + s_{0}f_{\gamma,W}\pi_{3},\\
\end{align*}
From these and~\eqref{equ:epsilon}, we get
\begin{align*}
f_{\alpha,V}\epsilon_{V}
&= (\iota_{0}f_{\alpha,U}r_{1}
  + \iota_{0}h_{\alpha}\pi_{1}
  + s_{0}f_{\alpha,W}\pi_{1})
  (\iota_{1}\epsilon_{U}r_{3}
  + s_{1}\epsilon_{W}\pi_{3}
  + \iota_{1}\sigma_{\epsilon}\pi_{3})\\
&= \iota_{0}(f_{\alpha,U}\epsilon_{U})r_{3} 
  + s_{0}(f_{\alpha,W}\epsilon_{W})\pi_{3}
  + \iota_{0}(f_{\alpha,U}\sigma_{\epsilon} + h_{\alpha}\epsilon_{W})\pi_{3},  
\shortintertext{and}
f_{\beta,V}\zeta_{V}
&= \iota_{0}(f_{\beta,U}\zeta_{U})r_{3} 
  + s_{0}(f_{\beta,W}\zeta_{W})\pi_{3}
  + \iota_{0}(f_{\beta,U}\sigma_{\zeta} + h_{\alpha}\zeta_{W})\pi_{3}.  
\end{align*}
From here,
\begin{align}
\begin{split}
\label{equ:final}
f_{\alpha,V}\epsilon_{V} + f_{\beta,V}\zeta_{V}
&= \iota_{0}(f_{\alpha,U}\epsilon_{U}+f_{\beta,U}\zeta_{U})r_{3}
  + s_{0}(f_{\alpha,W}\epsilon_{W}+f_{\beta,W}\zeta_{W})\pi_{3} \\
&\qquad + 
  \iota_{0}(
    f_{\alpha,U}\sigma_{\epsilon} + h_{\alpha}\epsilon_{W}
    + f_{\beta,U}\sigma_{\zeta} + h_{\beta}\zeta_{W}
    )\pi_{3}\\  
&= \iota_{0}f_{\gamma,U}r_{3}
  + s_{0}f_{\gamma,W}\pi_{3} 
  + \iota_{0}\Lambda\pi_{3},  
\end{split}
\end{align}
where
\[
\Lambda=f_{\alpha,U}\sigma_{\epsilon} + h_{\alpha}\epsilon_{W}
    + f_{\beta,U}\sigma_{\zeta} + h_{\beta}\zeta_{W}.
\]
Now, taking \eqref{equ:theta} into account,
\begin{align*}
\Lambda
&= f_{\alpha,U}\sigma_{\epsilon} 
  + f_{\beta,U}\sigma_{\zeta} 
  + h_{\alpha}\epsilon_{W}
  + h_{\beta}\zeta_{W} \\
&= \eta_{U}(\sigma_{\epsilon} + \sigma_{\zeta})
  + h_{\alpha}\epsilon_{W}
  + h_{\beta}\zeta_{W} \\
&= h_{\gamma}-h_{\alpha}\epsilon_{W}-h_{\beta}\zeta_{W}
  + h_{\alpha}\epsilon_{W}
  + h_{\beta}\zeta_{W} \\
&= h_{\gamma},  
\end{align*}
Substituting in~\eqref{equ:final}, we finally get
\[
f_{\alpha,V}\epsilon_{V} + f_{\beta,V}\zeta_{V}
= \iota_{0}f_{\gamma,U}r_{3}
  + s_{0}f_{\gamma,W}\pi_{3} 
  + \iota_{0}h_{\gamma}\pi_{3}
=f_{\gamma,V}.    
\]
This shows that $V$ is an object in $\mathcal{C}_{5}$, proving that the category is closed under extensions.

Since $\repFk$ is an abelian category, it follows that each $\mathcal{C}_{i}$ inherits an exact structure by taking as exact sequences those short exact sequences in $\repFk$ whose terms lie in $\mathcal{C}_{i}$ (see~\cite{Buh2010}). 
\end{proof}
\begin{coro}
The categories $\mathcal{A}_{1}=\repSk$, $\mathcal{A}_{2}=\repDk$, $\mathcal{A}_{3}=\repKk$, $\mathcal{A}_{4}=\repCk$, $\mathcal{A}_{5}=\lrel$, and $\mathcal{A}_{6}=\lreltt$ are exact.
\end{coro}
\begin{proof}
Exact structures are preserved under equivalence of categories.
\end{proof}
\begin{lema}
\label{lem:ffpreservesind}
A fully faithful functor $\func{G}\colon\mathcal{E}\to\mathcal{F}$ between Krull-Schmidt categories preserves and reflects indecomposability.
\end{lema}
\begin{proof}
Since $\func{G}$ is full and faithfull, we have an isomorphism $\Endo \simeq \Endo[(\func{G}U)]$, for each object $U$ in $\mathcal{E}$. Therefore, $U$ is indecomposable if and only if $\Endo$ is a local ring, which is equivalent, by the isomorphism, to $\Endo[(\func{G}U)]$ being a local ring. This, in turn, holds if and only if $\func{G}U$ is indecomposable
\end{proof}
\begin{coro}
The categories $\mathcal{A}_{1}$, $\mathcal{A}_{2}$, $\mathcal{A}_{3}$, $\mathcal{A}_{4}$, $\mathcal{A}_{5}$, and $\mathcal{A}_{6}$ are tame, of finite growth.
\end{coro}
\begin{proof}
Since the functors $\func{F}^{i}$ are fully faithful and the categories involved are Krull–Schmidt, Lemma~\ref{lem:ffpreservesind} implies that they preserve and reflect indecomposability. Therefore, each indecomposable object of $\mathcal{A}_i$ corresponds, via $\func{F}^{i}$, to an indecomposable object of $\repFk$, and this correspondence induces a bijection between the set of isomorphism classes of indecomposable objects in $\mathcal{A}_i$ and those in the full subcategory $\mathcal{C}_{i}$.

By Theorem~\ref{teo:finitegrowthF}, the algebra $A = k\carc{F}$ is tame of finite growth. In particular, for each dimension $d$, the set of indecomposable $A$-modules of dimension $d$ is contained in finitely many one-parameter families. Since each category $\mathcal{C}_{i}$ is a full subcategory of $\repFk$, its indecomposable objects form a subset of the indecomposable representations of $\carc{F}$. Hence, the number of one-parameter families of indecomposable objects of dimension $d$ in $\mathcal{C}_{i}$ is bounded by the corresponding number for $\repFk$.

It follows that each category $\mathcal{C}_{i}$, and therefore each $\mathcal{A}_i$, is tame of finite growth.
\end{proof}
The results of this section show that each of the six categories associated to the subproblems can be realized, via the functors $\func{F}^{i}$, as a full subcategory of $\repFk$. In particular, the study of their indecomposable objects reduces to the analysis of those indecomposable representations of the quiver $\carc{F}$ that satisfy certain structural constraints.

In order to describe these representations explicitly, in the next section we will introduce the matrix approach that will be used throughout the remainder of this paper. 

\section{Matrix presentations}
\label{sec:matrices}
We begin this section by presenting the notation that will be used in what follows.

For an integer $n\geq 1$, the matrices $\iup$, $\ido$ are the $(n+1)\times n$ matrices obtained by adjoining a row of zeros above, below, the identity matrix $I_{n}$. Analogously, the matrices $\iri$, $\ile$ are the $n\times (n+1)$ matrices obtained by adjoining a column of zeros to the right, to the left, of the identity matrix $I_{n}$. For $n=0$, the matrices $\iup[0]$ and $\ido[0]$ are equal and they are \textquote{formal} matrices having one row and zero columns, and representing the linear operator $0\to k$. The matrices $\iri[0]$ and $\ile[0]$ are also equal and are \textquote{formal} matrices having zero rows and one column, and representing the linear operator $k\to 0$.

We will also use the following standard canonical matrices. By $F_n(p^{s}(t))$ we will denote the Frobenius cell, also called rational canonical form cell, of order $n$ having the minimal polynomial $p^{s}(t)$, where $p(t)$ is monic and irreducible. In other words, $F_n(p^{s}(t))$ is the companion matrix of $p^{s}(t)$ (in particular, $n=s\cdot\deg p(t)$).

We denote by $J^{+}_{n}(0)$ the Jordan block of order $n$ with  eigenvalue $0$ and entries $1$ above the main diagonal.

The indecomposable representations of the tetrad can be interpreted as indecomposable representations of the quiver $\carc{F}$ in the following way. Let $V=(V_{0},V_{1},V_{2},V_{3},V_{4})$ be a representation of the tetrad $\carc{T}$, with dimension vector $d=(d_{0},d_{1},d_{2},d_{3},d_{4})$, and having matrix presentation
\[
M_{V}=
\begin{tikzpicture}[baseline={([yshift=-0.7ex]current bounding box.center)}]
\node[inner sep=0pt] (mat1)
  {\tetradmatnoline{tbt2}}; 
\end{tikzpicture}
\]
in reduced form. This representation corresponds to the representation 
\[
(W_{0},W_{1},W_{2},W_{3},W_{4},
f_{\alpha},f_{\beta},f_{\gamma},f_{\delta})
\] 
of the quiver $\carc{F}$, where $W_{i}=k^{d_{i}}$, for $i\in\{0,1,2,3,4\}$, and $[f_{\alpha}]=A$, $[f_{\beta}]=B$, $[f_{\delta}]=C$, and $[f_{\gamma}]=D$. Here the linear operators are expressed in matrix form, with respect to some fixed bases. Schematically, we have the following correspondence:
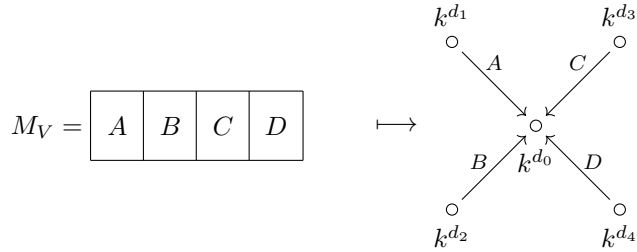
\[
M_{V}=
\begin{tikzpicture}[baseline={([yshift=-0.7ex]current bounding box.center)}]
\node[inner sep=0pt] (mat1)
  {\tetradmatnoline{tbt2}}; 

\node[copo,right=3cm of mat1.east,label={[label distance=3pt]below:$k^{d_{0}}$}] (V0) {};
\node[copo,above left=of V0,label={above:$k^{d_{1}}$}] (V1) {}; 
\node[copo,below left=of V0,label={below:$k^{d_{2}}$}] (V2) {}; 
\node[copo,above right=of V0,label={above:$k^{d_{3}}$}] (V3) {}; 
\node[copo,below right=of V0,label={below:$k^{d_{4}}$}] (V4) {}; 
\draw[->,corto] 
  (V1) -- 
  node[anchor=south,inner sep=5pt,font=\footnotesize] {$A$} 
  (V0);
\draw[->,corto] 
  (V2) -- 
  node[anchor=south,inner sep=5pt,font=\footnotesize,pos=0.3] {$B$} 
  (V0);
\draw[->,corto] 
  (V3) -- 
  node[anchor=south,inner sep=5pt,font=\footnotesize] {$C$} 
  (V0);
\draw[->,corto] 
  (V4) -- 
  node[anchor=south,inner sep=5pt,font=\footnotesize,pos=0.3] {$D$} 
  (V0);
\node at ( $ (mat1.east)!0.4!(V0) $ ) {$\longmapsto$};    
\end{tikzpicture}
\]
The classical subproblems discussed in Section~\ref{sec:FSPandsubproblems} can also be interpreted in matrix form. After fixing bases for the vector spaces involved, each representation can be written as a block matrix whose blocks correspond to the matrices of the linear operators.

In Table~\ref{tab:fourquivers} we collect the quivers associated with four of the subproblems under consideration and the corresponding block matrices of their representations.

We now describe the matrix presentations for linear relations.
\begin{table}[!ht]
\caption{Quivers associated with four of the subproblems considered, together with the corresponding matrix forms.  In the latter, the blocks $A$, $B$, $C$, and $D$ are the matrices of the linear operators $f_{\alpha}$,  $f_{\beta}$, $f_{\gamma}$, and $f_{\delta}$, respectively, with respect to fixed bases.}
\label{tab:fourquivers}
\[
\begin{array}{cc>{\hspace{1.2cm}}c}
\toprule
\text{Quiver} & \text{Representation} & \text{Matrix form}
\\
\cmidrule(lr){1-1}
\cmidrule(lr){2-2}
\cmidrule(lr){3-3}
% A3
\begin{tikzpicture}[node distance=1.3cm and 1.5cm,baseline=(current bounding box.center)]
\node[copo,label={left:$1$}] (S1) {};
\node[copo,below=of S1,label={left:$2$}] (S2) {};
\node[copo,right=of S1,label={right:$3$}] (S3) {};
\node[copo,below=of S3,label={right:$4$}] (S4) {};

\node at ([yshift=0.8cm]$(S1)!0.5!(S3)$ ) 
  {Quiver $\carc{S}$ (type $\widetilde{A}_{3}$)};

\draw[->,corto] 
  (S3) -- 
    node[anchor=south] {$\alpha$} 
  (S1);
\draw[->,corto] 
  (S3) -- 
    node[near end,fill=white] {$\beta$} 
  (S2);
\draw[->,corto] 
  (S4) -- 
    node[near end,fill=white] {$\gamma$} 
  (S1);
\draw[->,corto] 
  (S4) -- 
    node[anchor=north] {$\delta$} 
  (S2);
\end{tikzpicture}
  & (V_{1},V_{2},V_{3},V_{4},
     f_{\alpha},f_{\beta},f_{\gamma},f_{\delta})
   & \Smatrix[baseline=(current bounding box.center),name=t,remember picture]{|[text width=2.5em]|A}{|[text width=2.5em]|C}{|[text width=2.5em]|B}{|[text width=2.5em]|D}\rule[-1.5cm]{0pt}{0pt}
\begin{tikzpicture}[remember picture,overlay]
\draw[decoration={brace,raise=3pt,mirror},decorate]
  (mim-1-1.north west) --
    node[xshift=-5pt,anchor=east] {\footnotesize$\dim V_{1}$}
  (mim-1-1.south west);
\draw[decoration={brace,raise=3pt,mirror},decorate]
  (mim-2-1.north west) --
    node[xshift=-5pt,anchor=east] {\footnotesize$\dim V_{2}$}
  (mim-2-1.south west);
\draw[decoration={brace,raise=3pt},decorate]
  (mim-1-1.north west) --
    node[yshift=5pt,xshift=-2pt,anchor=south] {\footnotesize$\dim V_{3}$}
  (mim-1-1.north east);
\draw[decoration={brace,raise=3pt},decorate]
  (mim-1-2.north west) --
    node[yshift=5pt,xshift=2pt,anchor=south] {\footnotesize$\dim V_{4}$}
  (mim-1-2.north east);
\end{tikzpicture}  
\\
\midrule
% A2
\begin{tikzpicture}[node distance=1.3cm,baseline=(current bounding box.center)]
\node[copo,label={left:$1$}] (S1) {};
\node[copo,below=of S1,label={left:$2$}] (S2) {};
\node[copo,right=of S1,label={right:$3$}] (S3) {};

\node at ([yshift=0.8cm]$(S1)!0.5!(S3)$ ) 
  {Quiver $\carc{D}$ (type $\widetilde{A}_{2}$)};

\draw[->,corto] 
  (S3) -- 
    node[anchor=south] {$\alpha$} 
  (S1);
\draw[->,corto] 
  (S3) -- 
    node[below right] {$\beta$} 
  (S2);
\draw[->,corto] 
  (S2) -- 
    node[anchor=east] {$\gamma$} 
  (S1);
\end{tikzpicture}
  & (V_{1},V_{2},V_{3},
     f_{\alpha},f_{\beta},f_{\gamma})
   & \Dmatrix[baseline=(current bounding box.center),name=t,remember picture]{|[text width=2.5em]|A}{|[text width=2.5em]|C}{|[text width=2.5em]|B}\rule{0pt}{1.5cm}
\begin{tikzpicture}[remember picture,overlay]
\draw[decoration={brace,raise=3pt,mirror},decorate]
  (mim-1-1.north west) --
    node[xshift=-5pt,anchor=east] {\footnotesize$\dim V_{1}$}
  (mim-1-1.south west);
\draw[decoration={brace,raise=3pt,mirror},decorate]
  (mim-2-1.north west) --
    node[xshift=-5pt,anchor=east] {\footnotesize$\dim V_{2}$}
  (mim-2-1.south west);
\draw[decoration={brace,raise=3pt},decorate]
  (mim-1-1.north west) --
    node[yshift=5pt,xshift=-2pt,anchor=south] {\footnotesize$\dim V_{3}$}
  (mim-1-1.north east);
\draw[decoration={brace,raise=3pt},decorate]
  (mim-1-2.north west) --
    node[yshift=5pt,xshift=2pt,anchor=south] {\footnotesize$\dim V_{2}$}
  (mim-1-2.north east);
\end{tikzpicture}
\\
\midrule
% Kronecker
\begin{tikzpicture}[node distance=1.5cm,baseline=(current bounding box.center)]
\node[copo,label={left:$1$}] (K1) {};
\node[copo,right=of K1,label={right:$2$}] (K2) {};

\node at ([yshift=0.8cm]$(K1)!0.5!(K2)$ ) 
  {Kronecker quiver $\carc{K}$};

\draw[->,corto] 
  ( $ (K2) + (-3pt,3pt) $ ) -- 
    node[anchor=south] {$\alpha$} 
  ( $ (K1) + (3pt,3pt) $ );
\draw[->,corto] 
  ( $ (K2) + (-3pt,-3pt) $ ) -- 
    node[anchor=north] {$\beta$} 
  ( $ (K1) + (3pt,-3pt) $ );
\end{tikzpicture}
  & (V_{1},V_{2},f_{\alpha},f_{\beta})
  & \Kmatrix[baseline=(current bounding box.center),name=t,remember picture]{|[text width=2.5em]|A}{|[text width=2.5em]|B}\rule{0pt}{1.5cm}
\begin{tikzpicture}[remember picture,overlay]
\draw[decoration={brace,raise=3pt,mirror},decorate]
  (mim-1-1.north west) --
    node[xshift=-5pt,anchor=east] {\footnotesize$\dim V_{1}$}
  (mim-1-1.south west);
\draw[decoration={brace,raise=3pt},decorate]
  (mim-1-1.north west) --
    node[yshift=5pt,xshift=-2pt,anchor=south] {\footnotesize$\dim V_{2}$}
  (mim-1-1.north east);
\draw[decoration={brace,raise=3pt},decorate]
  (mim-1-2.north west) --
    node[yshift=5pt,xshift=2pt,anchor=south] {\footnotesize$\dim V_{2}$}
  (mim-1-2.north east);
\end{tikzpicture}
\\
\midrule
% Contragredient
\begin{tikzpicture}[node distance=1.5cm and 1.5cm,baseline=(current bounding box.center)]
\node[copo,label={left:$1$}] (C1) {};
\node[copo,right=of S1,label={right:$2$}] (C2) {};

\node[align=center] at ([yshift=1cm]$(K1)!0.5!(K2)$ ) 
  {Contragredient \\ Kronecker quiver $\carc{C}$};

\draw[->,corto] 
  ( $ (C2) + (-3pt,3pt) $ ) -- 
    node[anchor=south] {$\alpha$} 
  ( $ (C1) + (3pt,3pt) $ );
\draw[->,corto] 
  ( $ (C1) + (3pt,-3pt) $ ) -- 
    node[anchor=north] {$\beta$} 
  ( $ (C2) + (-3pt,-3pt) $ );
\end{tikzpicture}
  & (V_{1},V_{2},f_{\alpha},f_{\beta})
  & \CKmatrixi[baseline=(current bounding box.center),name=t,remember picture]{|[text width=2.5em]|A}{|[text width=2.5em]|B}\rule{0pt}{1.5cm}
\begin{tikzpicture}[remember picture,overlay]
\draw[decoration={brace,mirror,raise=3pt},decorate]
  (mim-1-2.north west) --
    node[xshift=-5pt,anchor=east] {\footnotesize$\dim V_{1}$}
  (mim-1-2.south west);
\draw[decoration={brace,raise=3pt},decorate]
  (mim-1-2.north west) --
    node[yshift=5pt,xshift=2pt,anchor=south] {\footnotesize$\dim V_{2}$}
  (mim-1-2.north east);

\draw[decoration={brace,mirror,raise=3pt},decorate]
  (mim-2-1.south west) --
    node[yshift=-5pt,anchor=north] {\footnotesize$\dim V_{1}$}
  (mim-2-1.south east);
\draw[decoration={brace,mirror,raise=3pt},decorate]
  (mim-2-1.north west) --
    node[xshift=-5pt,anchor=east] {\footnotesize$\dim V_{2}$}
  (mim-2-1.south west);
\end{tikzpicture}
\\[8ex]
\bottomrule
\end{array}
\]
\end{table}
For a finite-dimensional linear relation $\rho=(V_{1},V_{2},R)$ we can construct a matrix presentation $M_{\rho}$ by fixing ordered bases $B$ and $C$, for $V_{1}$ and $V_{2}$, respectively, and forming the matrix whose columns are the coordinates of a basis for $R$ (viewed as a subspace of $V_{1}\oplus V_{2}$) with respect to the basis 
\[
\{ (u_{i},0) \mid u_{i} \in B \} \cup \{ (0,v_{j}) \mid v_{j} \in C \}
\]
of $V_{1}\oplus V_{2}$. In the same way we can build matrix presentations for objects in the category $\lreltt$ of pairs of linear relations. These matrix forms are illustrated in Figure~\ref{fig:matrixlrellreltt}. 
\end{defi}
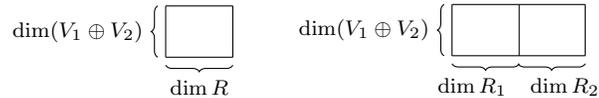
\begin{figure}[H]
\[
\begin{tikzpicture}[baseline={([yshift=-4pt]current bounding box.center)}]
\matrix[mimat,inner sep=0pt,text width=2.5em] (a)
{
\null \\
};
\draw (a-1-1.north west) -- (a-1-1.south west);
\draw (a-1-1.north east) -- (a-1-1.south east);
\draw (a-1-1.north west) -- (a-1-1.north east);
\draw (a-1-1.south west) -- (a-1-1.south east);
\draw[decoration={brace,raise=3pt,mirror},decorate]
  (a-1-1.north west) --
    node[xshift=-5pt,anchor=east] {\footnotesize$\dim (V_{1}\oplus V_{2})$}
  (a-1-1.south west);
\draw[decoration={brace,raise=3pt,mirror},decorate]
  (a-1-1.south west) --
    node[yshift=-5pt,anchor=north] {\footnotesize$\dim R$}
  (a-1-1.south east);
\end{tikzpicture}\qquad
\begin{tikzpicture}[baseline={([yshift=-4pt]current bounding box.center)}]
\matrix[mimat,inner sep=0pt,text width=2.5em] (a)
{
\null \& \null \\
};
\draw (a-1-1.north west) -- (a-1-1.south west);
\draw (a-1-1.north east) -- (a-1-1.south east);
\draw (a-1-2.north east) -- (a-1-2.south east);
\draw (a-1-1.north west) -- (a-1-2.north east);
\draw (a-1-1.south west) -- (a-1-2.south east);
\draw[decoration={brace,raise=3pt,mirror},decorate]
  (a-1-1.north west) --
    node[xshift=-5pt,anchor=east] {\footnotesize$\dim (V_{1}\oplus V_{2})$}
  (a-1-1.south west);
\draw[decoration={brace,raise=3pt,mirror},decorate]
  (a-1-1.south west) --
    node[xshift=-5pt,yshift=-5pt,anchor=north] {\footnotesize$\dim R_{1}$}
  (a-1-1.south east);
\draw[decoration={brace,raise=3pt,mirror},decorate]
  (a-1-2.south west) --
    node[xshift=5pt,yshift=-5pt,anchor=north] {\footnotesize$\dim R_{2}$}
  (a-1-2.south east);
\end{tikzpicture}
\]
\caption{Block structure of the matrix presentations of the linear relations $\rho=(V_{1},V_{2},R)$ (left) and $\sigma=(V_{1},V_{2},R_{1},R_{2})$ (right).}
\label{fig:matrixlrellreltt}
\end{figure}
\FloatBarrier

\begin{prop}
\label{pro:bijections}
For each $i\in\{1,2,3,4,5,6\}$, the functor $\func{F}^{i}$ induces a bijection between the set of isomorphism classes of indecomposable objects in the corresponding domain category and the subset $\Ind_{i}\carc{F}$ of $\IndF$ composed of all isomorphism classes of indecomposable objects having matrix presentations of the form specified in Table~\ref{tab:subproblemsmatrixform}.
\begin{table}[!ht]
\caption{The bijections induced by the six functors between indecomposable objects of the source category and the corresponding subset of $\Ind\carc{F}$. The rightmost column shows the associated matrix forms.}
\label{tab:subproblemsmatrixform}
\[
\begin{array}{cccc}
\toprule
\text{Functor} & \text{Domain category} & \text{Target subset}  & \text{Matrix form} 
\\
\cmidrule(lr){1-1}
\cmidrule(lr){2-2}
\cmidrule(lr){3-3}
\cmidrule(lr){4-4}
\func{F}^{1}
  & \repSk & \Ind_{1}\carc{F}
  & \tetradmat{i=I_{m},ii=0,iii=0,iv=I_{n},v=A,vi=B,vii=C,viii=D}{}{}
\\
\func{F}^{2}
  & \repDk & \Ind_{2}\carc{F}
  & \tetradmat{i=I_{m},ii=0,iii=0,iv=I_{n},v=A,vi=B,vii=C,viii=I_{n}}{}{}
\\
\func{F}^{3}
  & \repKk & \Ind_{3}\carc{F}
  & \tetradmat{i=I_{m},ii=0,iii=0,iv=I_{n},v=A,vi=I_{n},vii=B,viii=I_{n}}{}{}
\\
\func{F}^{4}
  & \repCk & \Ind_{4}\carc{F}
  & \tetradmat{i=I_{m},ii=0,iii=0,iv=I_{n},v=I_{m},vi=A,vii=B,viii=I_{n}}{}{}
\\
\func{F}^{5}
  & \lrel & \Ind_{5}\carc{F}
  & \tetradmatonerel{i=I_{n},ii=0,iii=0,iv=I_{n},v=I_{n},vi=I_{n},vii={},
    viii=\raisebox{0.8em}[0pt][0pt]{$M$}}{}{}
\\
\func{F}^{6}
  & \lreltt & \Ind_{6}\carc{F}
  & \tetradmattworel{i=I_{m},iv=I_{n},v={},
vi={\raisebox{0.8em}[0pt][0pt]{$M$}},
vii={},
viii={\raisebox{0.8em}[0pt][0pt]{$N$}}%
}{}{}
\\
\bottomrule  
\multicolumn{3}{@{}l}{\text{\footnotesize for some non-negative integers $m$ and $n$.}}
\end{array}
\]
\end{table}
\end{prop}
\begin{proof}
We will carry out the proof for $\func{F}^{1}$. For the other five cases the proofs are analogous. Consider the restriction of $\func{F}^{1}$ to $\IndS$. Since $\func{F}^{1}$ preserves indecomposability (see Lemma~\ref{lem:ffpreservesind}), the definition of $\func{F}^{1}$ guarantees that the image of $\func{F}^{1}$ is a subset of $\Ind_{1}\carc{F}$. To see that this restriction is injective, let us take two representations $V,W$ in $\IndS$ such that $\func{F}^{1}(V)\simeq\func{F}^{1}(W)$. According to Lemma~\ref{lem:ffpreservesiso}, the functor reflects isomorphism, so $V\simeq W$.

Let us, finally, consider any indecomposable representation in $\Ind_{1}\carc{F}$ , let us say
\[
V=
  \biggl(V_{0},V_{1},V_{2},V_{3},V_{4},
  \twomat{I_{m}}{0},\twomat{0}{I_{n}},
  \twomat{f_{\alpha}}{f_{\beta}},\twomat{f_{\gamma}}{f_{\delta}} 
  \biggr).
\]
From the form of the morphisms, we immediately obtain 
\[
V_{0}=
\Ima\twomat{I_{m}}{0}\oplus \Ima\twomat{0}{I_{n}}\simeq 
V_{1}\oplus V_{2}.
\]
Furthermore, $f_{\alpha}\colon V_{3}\to V_{1}$, $f_{\beta}\colon V_{3}\to V_{2}$, $f_{\gamma}\colon V_{4}\to V_{1}$, and $f_{\delta}\colon V_{4}\to V_{2}$, so 
\[
W=(V_{1},V_{2},V_{3},V_{4},
f_{\alpha},f_{\beta},f_{\gamma},f_{\delta})
\] 
is a well-defined object in $\repSk$, whose indecomposability is guaranteed by Lemma~\ref{lem:ffpreservesind}. Additionally, $\func{F}^{1}(W)\simeq V$ so the restriction is indeed a surjection and our claim is proved. 
\end{proof}
As a consequence of the previous result, we get the solution for the six subproblems considered.
\begin{coro}
For an arbitrary field $k$, we have the following:
\begin{enumerate}
\item\label{teo:indecomquivers}
All indecomposable representations for the quivers $\carc{S}$, $\carc{D}$, $\carc{K}$, and $\carc{C}$ are exhausted, up to isomorphism, by the representations given in matrix form in Figures~\ref{fig:indecomposableS}, \ref{fig:indecomposableD}, \ref{fig:indecomposableK}, and \ref{fig:indecomposableCK}, respectively.

\item\label{teo:linearrelations}
Given $V_{1}$, $V_{2}$, two finite dimensional vector spaces over $k$, all indecomposable linear relations $(V_{1},R)$ on $V_{1}$ are exhausted, up to isomorphism, by the four types of matrix presentations given in Figure~\ref{fig:onelinear}, and all indecomposable pairs of linear relations $(V_{1},V_{2},R_{1},R_{2})$  are exhausted, up to isomorphism, by the seven types of matrix presentations given in Figure~\ref{fig:twolinear}.
\end{enumerate}
\end{coro}
\begin{proof}
For each particular case all we have to do, according to Proposition~\ref{pro:bijections}, is to consider from the list in Figure~\ref{fig:FSPsolution}, those matrix presentations which have the specified form. This immediately leads to the representations listed in our statement.
\end{proof}

\subsection*{Explicit classification of indecomposables}
%\label{sec:indecomposables}
The mechanism underlying the classifications shown in this section can be schematically summarized as follows:
\[
\Ind(\repFk)\longrightarrow
\Ind_{i}\carc{F}\longrightarrow
\Ind(\mathcal{A}_{i}),\qquad
\text{for all $i\in\{1,2,3,4,5,6\}$.}
\]
Here, $\Ind(\repFk)$ is the set of isomorphism classes of indecomposable representations of the quiver $\carc{F}$, whose matrix presentations are listed in Figure~\ref{fig:FSPsolution}. For each $i\in\{1,2,3,4,5,6\}$, the subset $\Ind_{i}\carc{F}$ consists of those indecomposable representations whose matrix presentations have the structural form described in Table~\ref{tab:subproblemsmatrixform}. Proposition~\ref{pro:bijections} shows that the functor $\func{F}^{i}$ induces a bijection between $\Ind(\mathcal{A}_{i})$ and $\Ind_{i}\carc{F}$. The classification of indecomposable objects in $\mathcal{A}_{i}$ is obtained, therefore, by selecting from Figure~\ref{fig:FSPsolution} those matrix presentations belonging to $\Ind_{i}\carc{F}$.
\begin{figure}[H]
\centering%\footnotesize
\renewcommand\arraystretch{1.2}
\begin{tabular}{%
  |>{\centering\arraybackslash}p{%
    \dimexpr0.5\textwidth-2\tabcolsep-1.5\arrayrulewidth\relax}|
  >{\centering\arraybackslash}p{%
    \dimexpr0.5\textwidth-2\tabcolsep-1.5\arrayrulewidth\relax}|
}
\multicolumn{2}{c}{\normalfont Regular:} \\
\hline
\multicolumn{2}{|>{\centering\arraybackslash}p{%
    \dimexpr\textwidth-2\tabcolsep-3\arrayrulewidth\relax\relax}|}{%
    \rlap{$\mathrm{0}=\mathrm{0}^{\ast}$}%\par
  \tetradmatinfo{
    i=I_{n},iv=I_{n},v=I_{n},
    vi=I_{n},vii={|[text width=4.5em]|F_{n}(p^{s}(t))},
    viii={|[text width=4.5em]|I_{n}},matname=A}%
    {remember picture}{}%
    {1}{0}{(n,n,n,n)}
  \begin{tikzpicture}[remember picture,overlay]
    \node[anchor=west,overlay,outer sep=0pt,inner sep=0pt] at (A.east)
    {, $p(t)\neq t$};
  \end{tikzpicture}%
}%
\\
\hline
\rlap{$\mathrm{I}=\mathrm{I}^{\ast}$}\par
\tetradmatinfo{i=I_{n},iii=I_{n},iv=I_{n},
  vi=I_{n},vii={|[text width=2.5em]|J_{n}^{+}(0)},
  viii={|[text width=2.5em]|I_{n}}}{}{}{1}{0}{(n,n,n,n)}
&
\rlap{$\mathrm{II}=\mathrm{II}^{\ast}$}\par
\tetradmatinfo{i=I_{n+1},iii=I_{n+1},iv=\ile,
  v=\ido,vi=I_{n},vii=0,viii=I_{n}}%
  {}{}{0}{1}{(n+1,n+1,n,n)}
\\
\hline
\end{tabular}%
\par\medskip
\begin{tabular}{%
  |>{\centering\arraybackslash}p{%
    \dimexpr0.5\textwidth-2\tabcolsep-1.5\arrayrulewidth\relax}|
  >{\centering\arraybackslash}p{%
    \dimexpr0.5\textwidth-2\tabcolsep-1.5\arrayrulewidth\relax}|
}
\multicolumn{2}{c}{\normalfont Non-regular:} 
\\
\hline
\rlap{$\mathrm{III}$}\par
\tetradmatinfo{i=I_{n+1},iv=I_{n},v=\iup,
  vi=I_{n},vii=\ido,viii=I_{n}}%
  {}{}{0}{1}{(n+1,n,n,n)}
&
\rlap{$\mathrm{III}^{\ast}$}\par
\tetradmatinfo{i=I_{n},iv=I_{n+1},v=\ile,
  vi=I_{n+1},vii=\iri,viii=I_{n+1}}%
  {}{}{0}{1}{(n,n+1,n+1,n+1)}
\\
\hline
\rlap{$\mathrm{IV}$}\par
\tetradmatinfo{i=I_{n+1},iv=I_{n+1},v=I_{n+1},
  vi=I_{n+1},vii=\iup,viii=\ido}%
  {}{}{0}{1}{(n+1,n+1,n+1,n)}
&
\rlap{$\mathrm{IV}^{\ast}$}\par
\tetradmatinfo{i=I_{n+1},iv=I_{n+1},v=I_{n+1},
  vi=I_{n+1},vii=\ile[n+1],viii=\iri[n+1]}%
  {}{}{0}{1}{(n+1,n+1,n+1,n+2)}
\\
\hline
\rlap{$\mathrm{V}$}\par
\tetradmatinfov{i=I_{n},iii=\lrz,v=I_{n},vi=\lrz,
  vii={J^{+}_{n}(0)},viii=I_{n},ix=\lroz,
  x=I_{n},xi={J^{+}_{n}(0)},xii=\lroz}%
  {}{text width=2.7em}{0}{1}{(n+1,n+1,n+1,n)}
&
\rlap{$\mathrm{V}^{\ast}$}\par
\tetradmatinfov{i=\ile,iii=\lroz,iv=\ile,v=\iri,vi=\lroz,
  vii=\iri,viii=\ile,ix=\lroz,
  xi=\ile,xii=\lroz}%
  {}{text width=2.7em}{0}{1}{(n+1,n+1,n+1,n+1)}
\\
\hline
\end{tabular}%
\caption{Classification of all indecomposable representations of a tetrad, up to isomorphism and permutation of the subspaces.}
\label{fig:FSPsolution}%
\end{figure}

\begin{figure}
\[
\renewcommand\arraystretch{1.2}
\begin{tabular}{%
  |>{\centering\arraybackslash}p{%
    \dimexpr0.4\textwidth-2\tabcolsep-1.5\arrayrulewidth\relax}|
  >{\centering\arraybackslash}p{%
    \dimexpr0.4\textwidth-2\tabcolsep-1.5\arrayrulewidth\relax}|
}
\multicolumn{2}{c}{\normalfont Regular:} \\
\hline
\multicolumn{2}{|>{\centering\arraybackslash}p{%
    \dimexpr0.8\textwidth\relax}|}{\rlap{$\mathrm{0}=\mathrm{0}^{\ast}$}%\par
        \tttmatrixsimple[remember picture,baseline=(current bounding box.center)]{I_{n}}{|[text width=4.5em]|F_{n}(p^{s}(t))}{I_{n}}{|[text width=4.5em]|I_{n}}{1}{0}{(n,n,n,n)}}%
\begin{tikzpicture}[remember picture,overlay]
\node[align=center,anchor=west,outer sep=0pt,inner sep=0pt] at ([xshift=3pt]mim.east) 
{, $p(t)\neq t$};
\end{tikzpicture}% 
\\
\hline
\rlap{$\mathrm{I}=\mathrm{I}^{\ast}$}\par
\tttmatrixsimple[baseline=(current bounding box.center)]{I_{n}}{J_{n}^{+}(0)}{I_{n}}{I_{n}}{1}{0}{(n,n,n,n)}
&
\rlap{$\mathrm{II}=\mathrm{II}^{\ast}$}\par
\tttmatrixsimple[baseline=(current bounding box.center)]{I_{n+1}}{\ido}{\ile}{I_{n}}{0}{1}{(n+1,n,n+1,n)}
\\
\hline
\end{tabular}
\]
\[
\renewcommand\arraystretch{1.2}
\begin{tabular}{%
  |>{\centering\arraybackslash}p{%
    \dimexpr0.4\textwidth-\tabcolsep-1.5\arrayrulewidth\relax}|
  >{\centering\arraybackslash}p{%
    \dimexpr0.4\textwidth-\tabcolsep-1.5\arrayrulewidth\relax}|
}
\multicolumn{2}{c}{\normalfont Non-regular:} \\
\hline
\rlap{$\mathrm{III}$}\par
\tttmatrixsimple[baseline=(current bounding box.center)]{\iup}{\ido}{I_{n}}{I_{n}}{0}{1}{(n+1,n,n,n)}
&
\rlap{$\mathrm{III}^{\ast}$}\par
\tttmatrixsimple[baseline=(current bounding box.center)]{\ile}{\iri}{I_{n+1}}{I_{n+1}}{0}{1}{(n,n+1,n+1,n+1)} 
\\
\hline
\rlap{$\mathrm{IV}$}\par
\tttmatrixsimple[baseline=(current bounding box.center)]{I_{n+1}}{\iup}{I_{n+1}}{\ido}{0}{1}{(n+1,n+1,n+1,n)}
&
\rlap{$\mathrm{IV}^{\ast}$}\par
\tttmatrixsimple[baseline=(current bounding box.center)]{I_{n}}{\ile}{I_{n}}{\iri}{0}{1}{(n,n,n,n+1)}
\\
\hline
\end{tabular}
\]
\caption{Indecomposable representations of the quiver $\carc{S}$ of type $\widetilde{A}_{3}$.}
\label{fig:indecomposableS}
\end{figure}

\begin{figure}
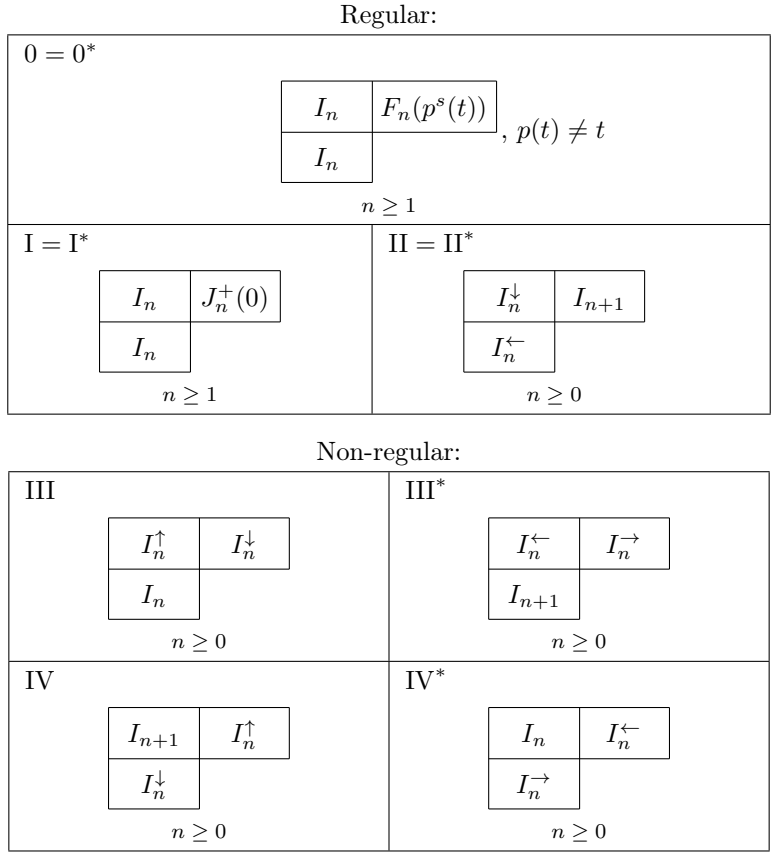

\[
\renewcommand\arraystretch{1.2}
\begin{tabular}{%
  |>{\centering\arraybackslash}p{%
    \dimexpr0.4\textwidth-2\tabcolsep-1.5\arrayrulewidth\relax}|
  >{\centering\arraybackslash}p{%
    \dimexpr0.4\textwidth-2\tabcolsep-1.5\arrayrulewidth\relax}|
}
\multicolumn{2}{c}{\normalfont Regular:} \\
\hline
\multicolumn{2}{|>{\centering\arraybackslash}p{%
    \dimexpr0.8\textwidth\relax}|}{\rlap{$\mathrm{0}=\mathrm{0}^{\ast}$}%\par
    \medskip 
  \tripmatrix[remember picture,baseline=(current bounding box.center)]{I_{n}}{|[text width=4.7em]|F_{n}(p^{s}(t))}{I_{n}}{1}{T1}{0}{(n,n,n)}%
  \tikz[remember picture,overlay]{%
    \node[anchor=west,outer sep=0pt,inner sep=0pt] at (T1.east)
  {\,, $p(t)\neq t$};
  }%
}   
\\
\hline
\rlap{$\mathrm{I}=\mathrm{I}^{\ast}$}\par\medskip 
  \tripmatrix[remember picture,baseline=(current bounding box.center)]{I_{n}}{J ^{+}_{n}(0)}{I_{n}}{1}{T2}{0}{(n,n,n)}%
&
  \rlap{$\mathrm{II}=\mathrm{II}^{\ast}$}\par\medskip 
  \tripmatrix[remember picture,baseline=(current bounding box.center)]{\ido}{I_{n+1}}{\ile}{0}{T3}{1}{(n+1,n,n+1)}
\\
\hline
\end{tabular}
\]
\[
\renewcommand\arraystretch{1.2}
\begin{tabular}{%
  |>{\centering\arraybackslash}p{%
    \dimexpr0.4\textwidth-\tabcolsep-1.5\arrayrulewidth\relax}|
  >{\centering\arraybackslash}p{%
    \dimexpr0.4\textwidth-\tabcolsep-1.5\arrayrulewidth\relax}|
}
\multicolumn{2}{c}{\normalfont Non-regular:} \\
\hline
\rlap{$\mathrm{III}$}\par\medskip 
  \tripmatrix[remember picture,baseline=(current bounding box.center)]{\iup}{\ido}{I_{n}}{0}{T4}{1}{(n+1,n,n)}%
&
  \rlap{$\mathrm{III}^{\ast}$}\par\medskip 
  \tripmatrix[remember picture,baseline=(current bounding box.center)]{\ile}{\iri}{I_{n+1}}{0}{T5}{1}{(n.n+1,n+1)}%
\\
\hline
  \rlap{$\mathrm{IV}$}\par\medskip 
  \tripmatrix[remember picture,baseline=(current bounding box.center)]{I_{n+1}}{\iup}{\ido}{0}{T6}{1}{(n+1,n+1,n)}%
&
  \rlap{$\mathrm{IV}^{\ast}$}\par\medskip 
  \tripmatrix[remember picture,baseline=(current bounding box.center)]{I_{n}}{\ile}{\iri}{0}{T7}{1}{(n,n,n+1)}%
\\
\hline
\end{tabular}
\]
\caption{Indecomposable representations of the quiver $\carc{D}$ of type $\widetilde{A}_{2}$.}
\label{fig:indecomposableD}
\end{figure}

\begin{figure}
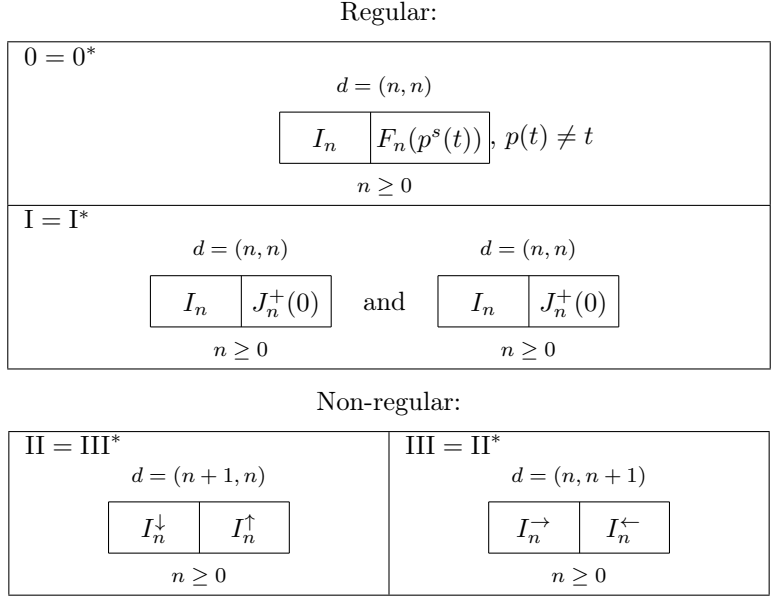

\[
\begin{tabular}{|>{\centering\arraybackslash}p{.8\textwidth}|}
\multicolumn{1}{c}{\normalfont Regular:} \\[1ex]
\hline
\rlap{$\mathrm{0}=\mathrm{0}^{\ast}$}%\par 
  \hormatrixl[remember picture,baseline=(current bounding box.center)]{I_{n}}{|[text width=4.5em]|F_{n}(p^{s}(t))}{0}{0}{(n,n)}{K1}
  \tikz[remember picture,overlay]{
  \node[anchor=west,outer sep=0pt,inner sep=0pt] at (K1.east)
  {, $p(t)\neq t$};
  }
\\
\hline
\rlap{$\mathrm{I}=\mathrm{I}^{\ast}$}\par 
  \hormatrixl[remember picture,baseline=(current bounding box.center)]{I_{n}}{J^{+}_{n}(0)}{0}{0}{(n,n)}{K2a}% 
\qquad%
\qquad% 
\hormatrixl[remember picture,baseline=(current bounding box.center)]{I_{n}}{J^{+}_{n}(0)}{0}{0}{(n,n)}{K2b}
  \tikz[remember picture,overlay]{
  \node at ( $ (K2a.east)!0.5!(K2b.west) $){and};
  }
\\
\hline  
\end{tabular}
\]
\[
\begin{tabular}{|>{\centering\arraybackslash}%
  p{\dimexpr0.4\textwidth-2\fboxsep-1.5\fboxrule\relax}|
>{\centering\arraybackslash}%
  p{\dimexpr0.4\textwidth-2\fboxsep-1.5\fboxrule\relax}|
}
\multicolumn{2}{c}{\normalfont Non-regular:} \\[1ex]
\hline
\rlap{$\mathrm{II}=\mathrm{III}^{\ast}$}\par
  \hormatrixl[remember picture,baseline=(current bounding box.center)]{\ido}{\iup}{0}{1}{(n+1,n)}{K3}
  &
\rlap{$\mathrm{III}=\mathrm{II}^{\ast}$}
  \hormatrixl[remember picture,baseline=(current bounding box.center)]{\iri}{\ile}{0}{1}{(n,n+1)}{K4} \\
\hline   
\end{tabular}
\]
\caption{Indecomposable representations of the Kronecker quiver $\carc{K}$.}
\label{fig:indecomposableK}
\end{figure}

\begin{figure}
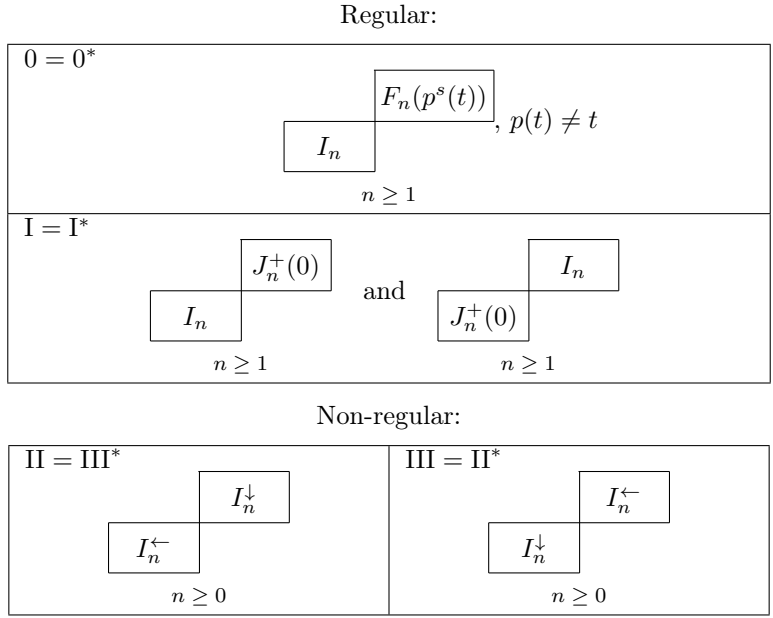

\[
\begin{tabular}{|>{\centering\arraybackslash}p{0.8\textwidth}|}
\multicolumn{1}{c}{\normalfont Regular:} \\[1ex]
\hline
\rlap{$\mathrm{0}=\mathrm{0}^{\ast}$}%\par 
  \diamatrix[remember picture,baseline=(current bounding box.center)]{|[text width=4.5em]|F_{n}(p^{s}(t))}{I_{n}}{1}{0}{(n,n)}{CK1}%
\tikz[remember picture,overlay]{%
  \node[anchor=west,outer sep=0pt,inner sep=0pt] at (CK1.east)
  {, $p(t)\neq t$};
}%
\\
\hline
\rlap{$\mathrm{I}=\mathrm{I}^{\ast}$}\par 
  \diamatrix[remember picture,baseline=(current bounding box.center)]{J^{+}_{n}(0)}{I_{n}}{1}{0}{(n,n)}{CK2a}% 
  \qquad\qquad% 
  \diamatrix[remember picture,baseline=(current bounding box.center)]{I_{n}}{J^{+}_{n}(0)}{1}{0}{(n,n)}{CK2b}
\tikz[remember picture,overlay]{\node at ( $ (CK2a.east)!0.5!(CK2b.west) $ )
  {and};
}%
\\
\hline
\end{tabular}
\]
\[
\begin{tabular}{|>{\centering\arraybackslash}%
  p{\dimexpr0.4\textwidth-2\fboxsep-1.5\fboxrule\relax}|
>{\centering\arraybackslash}%
  p{\dimexpr0.4\textwidth-2\fboxsep-1.5\fboxrule\relax}|
}
\multicolumn{2}{c}{\normalfont Non-regular:} \\[1ex]
\hline
\rlap{$\mathrm{II}=\mathrm{III}^{\ast}$}\par 
  \diamatrix[remember picture,baseline=(current bounding box.center)]{\ido}{\ile}{0}{1}{(n,n+1)}{CK3}
  &
  \rlap{$\mathrm{III}=\mathrm{II}^{\ast}$}
  \diamatrix[remember picture,baseline=(current bounding box.center)]{\ile}{\ido}{0}{1}{(n+1,n)}{CK4}
\\
\hline
\end{tabular}
\]
\caption{Indecomposable representations of the contragredient Kronecker quiver $\carc{C}$.}
\label{fig:indecomposableCK}
\end{figure}

\begin{figure}
\[
\renewcommand\arraystretch{1.2}
\begin{tabular}{%
  |>{\centering\arraybackslash}p{%
    \dimexpr0.4\textwidth-\tabcolsep-1.5\arrayrulewidth\relax}|
  >{\centering\arraybackslash}p{%
    \dimexpr0.4\textwidth-\tabcolsep-1.5\arrayrulewidth\relax}|
}
\multicolumn{2}{c}{\normalfont Regular:} \\
\hline
\rlap{$\mathrm{0}=\mathrm{0}^{\ast}$}\par\medskip 
  \vermatrix[remember picture,baseline=(current bounding box.center)]{|[text width=4.5em]|F_{n}(p^{s}(t))}{|[text width=4.5em]|I_{n}}{1}{LR1}    \tikz[remember picture,overlay]{%
  \node[anchor=west,overlay,outer sep=0pt,inner sep=0pt] at (LR1.east)
    {, $p(t)\neq t$};
    }%
&
  \rlap{$\mathrm{I}=\mathrm{I}^{\ast}$}\par\medskip 
  \vermatrix[baseline=(current bounding box.center)]{J^{+}_{n}(0)}{I_{n}}{1}{LR2}
\\
\hline
\end{tabular}
\]
\[
\renewcommand\arraystretch{1.2}
\begin{tabular}{%
  |>{\centering\arraybackslash}p{%
    \dimexpr0.4\textwidth-\tabcolsep-1.5\arrayrulewidth\relax}|
  >{\centering\arraybackslash}p{%
    \dimexpr0.4\textwidth-\tabcolsep-1.5\arrayrulewidth\relax}|
}
\multicolumn{2}{c}{\normalfont Non-regular:} \\
\hline

\rlap{$\mathrm{II}=\mathrm{III}^{\ast}$}\par\medskip 
  \vermatrix[baseline=(current bounding box.center)]{\iup}{\ido}{0}{LR3}%
&
  \rlap{$\mathrm{II}=\mathrm{III}^{\ast}$}\par\medskip 
  \vermatrix[baseline=(current bounding box.center)]{\ile}{\iri}{0}{LR4} 
\\
\hline
\end{tabular}
\]
\caption{Indecomposable representations for one linear relation.}
\label{fig:onelinear}
\end{figure}

\begin{figure}
\[
\renewcommand\arraystretch{1.2}
\begin{tabular}{%
  |>{\centering\arraybackslash}p{%
    \dimexpr0.4\textwidth-2\tabcolsep-1.5\arrayrulewidth\relax}|
  >{\centering\arraybackslash}p{%
    \dimexpr0.4\textwidth-2\tabcolsep-1.5\arrayrulewidth\relax}|
}
\multicolumn{2}{c}{\normalfont Regular:} \\
\hline
\multicolumn{2}{|>{\centering\arraybackslash}p{%
    \dimexpr0.8\textwidth\relax}|}{\rlap{$\mathrm{0}=\mathrm{0}^{\ast}$}%\par
    \medskip 
  \tvermatrix[remember picture,baseline=(current bounding box.center)]{I_{n}}{|[text width=4.5em]|F_{n}(p^{s}(t))}{I_{n}}{|[text width=4.5em]|I_{n}}{1}{2LR1}%
  \tikz[remember picture,overlay]{
    \node[anchor=west,outer sep=0pt,inner sep=0pt] at (2LR1.east)
  {\,, $p(t)\neq t$};}%
}   
\\
\hline
\rlap{$\mathrm{I}=\mathrm{I}^{\ast}$}\par\medskip 
  \tvermatrix[remember picture,baseline=(current bounding box.center)]{I_{n}}{J ^{+}_{n}(0)}{I_{n}}{I_{n}}{1}{2LR2}%
&
  \rlap{$\mathrm{II}=\mathrm{II}^{\ast}$}\par\medskip 
  \tvermatrix[remember picture,baseline=(current bounding box.center)]{I_{n+1}}{\ido}{\ile}{I_{n}}{0}{2LR3}
\\
\hline
\end{tabular}
\]
\[
\renewcommand\arraystretch{1.2}
\begin{tabular}{%
  |>{\centering\arraybackslash}p{%
    \dimexpr0.4\textwidth-\tabcolsep-1.5\arrayrulewidth\relax}|
  >{\centering\arraybackslash}p{%
    \dimexpr0.4\textwidth-\tabcolsep-1.5\arrayrulewidth\relax}|
}
\multicolumn{2}{c}{\normalfont Non-regular:} \\
\hline
\rlap{$\mathrm{III}$}\par\medskip 
  \tvermatrix[remember picture,baseline=(current bounding box.center)]{\iup}{\ido}{I_{n}}{I_{n}}{0}{2LR4}%
&
  \rlap{$\mathrm{III}^{\ast}$}\par\medskip 
  \tvermatrix[remember picture,baseline=(current bounding box.center)]{\ile}{\iri}{I_{n+1}}{I_{n+1}}{0}{2LR5}%
\\
\hline
  \rlap{$\mathrm{IV}$}\par\medskip 
  \tvermatrix[remember picture,baseline=(current bounding box.center)]{I_{n+1}}{\iup}{I_{n+1}}{\ido}{0}{2LR6}%
&
  \rlap{$\mathrm{IV}^{\ast}$}\par\medskip 
  \tvermatrix[remember picture,baseline=(current bounding box.center)]{I_{n}}{\ile}{I_{n}}{\iri}{0}{2LR7}%
\\
\hline
\end{tabular}
\]
\caption{Indecomposable representations for two linear relations.}
\label{fig:twolinear}
\end{figure}

\FloatBarrier

\printbibliography

\end{document}